\documentclass{article}

\usepackage{microtype}
\usepackage{graphicx}
\usepackage{subfigure}
\usepackage{subfloat}
\usepackage{booktabs} 
\usepackage[OT1]{fontenc} 

\usepackage{algorithmic}
\usepackage[tbtags]{amsmath}
\usepackage{amssymb}
\usepackage{amsthm}
\usepackage{bbm}
\usepackage{xcolor}
\usepackage{mdframed}
\usepackage{hyperref}

\usepackage[accepted]{icml2022}

\newtheorem{lemma}{Lemma}
\newtheorem{corollary}{Corollary}
\newtheorem{theorem}{Theorem}
\newtheorem{definition}{Definition}
\newtheorem{proposition}{Proposition}

\icmltitlerunning{Provable Acceleration of Heavy Ball beyond Quadratics}

\def\reals{\mathbb{R}}

\def\A{\mathcal{A}}

\newcommand{\XX}{\mathcal{X}}

\newcommand{\ao}{\textsc{average out }}
\newcommand{\av}[1]{$#1$-\textsc{average out}}
\newcommand{\avg}[2]{$#1$-\textsc{average out} w.r.t.\text{ }$#2$}
\newcommand{\jk}[1]{{\bf \color{red}{[JW: #1]}}}

\begin{document}

\twocolumn[
\icmltitle{Provable Acceleration of Heavy Ball beyond Quadratics for a Class of Polyak-\L{}ojasiewicz Functions when the Non-Convexity is Averaged-Out}



\icmlsetsymbol{equal}{*}

\begin{icmlauthorlist}
\icmlauthor{Jun-Kun Wang}{yale}
\icmlauthor{Chi-Heng Lin}{gt}
\icmlauthor{Andre Wibisono}{yale}
\icmlauthor{Bin Hu}{uiuc}
\end{icmlauthorlist}

\icmlaffiliation{yale}{Computer Science, Yale University}
\icmlaffiliation{gt}{Electrical and Computer Engineering, Georgia Institute of Technology}
\icmlaffiliation{uiuc}{Electrical and Computer Engineering \& Coordinated Science Laboratory, University of Illinois at Urbana-Champaign}

\icmlcorrespondingauthor{Jun-Kun Wang}{jun-kun.wang@yale.edu}
\icmlcorrespondingauthor{Chi-Heng Lin}{clin354@gatech.edu}
\icmlcorrespondingauthor{Andre Wibisono}{andre.wibisono@yale.edu}
\icmlcorrespondingauthor{Bin Hu}{binhu7@illinois.edu}

\icmlkeywords{Machine Learning, ICML}

\vskip 0.3in
]



\printAffiliationsAndNotice{}  


\begin{abstract}
Heavy Ball (HB) nowadays is one of the most popular momentum methods in non-convex 
optimization. It has been widely observed that incorporating the Heavy Ball dynamic in gradient-based methods accelerates the training process of modern machine learning models. However, the progress on establishing its theoretical foundation of acceleration is apparently far behind its empirical success. 
Existing provable acceleration results are of the quadratic or close-to-quadratic functions, as the current techniques of showing HB's acceleration are limited to the case when the Hessian is fixed. In this work, we develop some new techniques that help show acceleration beyond quadratics, which is achieved by analyzing how the change of the Hessian at two consecutive time points affects the convergence speed. Based on our technical results, a class of Polyak-\L{}ojasiewicz (PL) optimization problems for which provable acceleration can be achieved via HB is identified. 
Moreover, our analysis demonstrates a benefit of adaptively setting the momentum parameter.\\
\textbf{(08/29/2023):
Erratum is added in Appendix~\ref{erratum}.
This is an updated version that fixes an issue in the previous version. 
An additional condition needs to be satisfied for
the acceleration result of HB beyond quadratics in this work, 
which naturally holds
when the dimension is one or, more broadly, when the Hessian is diagonal. 
We elaborate on the issue in Appendix~\ref{erratum}.
} 
\end{abstract}

\section{Introduction}

Heavy Ball (a.k.a. Polyak's momentum) \citep{P64}
has become a dominant momentum method in various machine learning applications,
see e.g., \citep{Rnet16,KSH12,WRSSR17}. Many recently proposed algorithms
incorporate Heavy Ball momentum in lieu of Nesterov's momentum \citep{N13}, e.g.,
Adam \citep{KB15}, AMSGrad \citep{RKK18}, and a normalized momentum method \cite{CM21}, to name just a few.
Despite its success in practice, theoretical results of acceleration via Heavy Ball are sparse in the literature. Indeed, the strongly convex quadratic problems seem to be one of very few cases for which an acceleration via HB is shown (e.g., \citet{LRP16,GFJ15,LR17,PP21,CGZ19,SP20,NB15,DJ19,SDJS18,wang2018acceleration,WCA20,LGY20,XCHZ20,JKKNS18}). Recently, \citep{WLA21} provably show that HB converges faster than vanilla gradient descent (GD) for training an over-parametrized ReLU network and a deep linear network.
They show an accelerated linear rate that has a square root dependency on the conditional number of a certain Gram matrix, which improves the linear rate of GD for those problems. They develop a modular analysis of acceleration when the HB dynamic meets certain conditions. The conditions essentially require that an underlying Gram matrix or the Hessian does not deviate too much from that at the beginning.
Therefore, their theorem might not be applied to the case when the iterate does move far away from the initial point.
This limitation to acceleration, to our knowledge, appears in all the related works of the discrete-time HB. What is lacking in the HB literature is a provable acceleration when the Hessian can deviate significantly from that at the beginning  during the optimization process.

\begin{algorithm}[t]
\begin{algorithmic}[1]
\small
\caption{Heavy Ball (Equivalent Version 1) } 
\label{alg:HB2}
\STATE Required: the step size $\eta$ and the momentum parameter $\beta_{t} \in [0,1]$
\STATE Init: $w_{0} = w_{-1} \in \reals^d $
\FOR{$t=0$ to $T$}
\STATE Given current iterate $w_t$, compute gradient $\nabla f(w_t)$.
\STATE Update iterate $w_{t+1} = w_t - \eta \nabla f(w_t) + \beta_t ( w_t - w_{t-1} )$.
\ENDFOR
\end{algorithmic}
\vspace{-0.03in}
\end{algorithm}

This work aims at filling the gap. We develop a technical result that proves acceleration of HB when the Hessian can change substantially under certain conditions. Specifically, our analysis shows that HB has acceleration compared to GD for a class of Polyak-\L{}ojasiewicz problems when the non-convexity is averaged-out (to be elaborated soon).
The PL condition \citep{P63,KNS16} is sufficient 
to show a global linear convergence rate for GD without the need of assuming strong convexity. 
The notion of PL has been discovered in various non-convex problems over the past few years, e.g., computing the Wasserstein-barycenter of Gaussian distributions \citep{ACGS21}, minimizing a Wasserstein distance using point clouds \citep{MSS21}, training over-parametrized neural networks with smooth activation functions \citep{OS19,LZB21}, and sparse optimization on a certain measure \citep{C21}.
However, as far as we are aware, momentum methods (not necessarily HB) haven't been shown to converge faster than GD under PL in discrete time. Indeed, \citet{KDP18} show a global linear convergence of the discrete-time HB under PL, but the rate is not better than that of GD, c.f., \citep{KNS16}. 
While in continuous time, some linear rates of HB under PL are established, e.g., \citep{ADR20,AGV21}, it is not clear if the rates 
do show a benefit of the use of HB compared to GD, 
and is neither clear if the result can be carried over in the discrete time. In convex optimization, it is known that Nesterov's momentum and HB share the same ODE in continuous time \citep{SDJS18}. Yet, the 
acceleration disappears when one discretizes the dynamic of HB and bounds the discretization error.
To our knowledge, provably showing any benefit of discrete-time momentum under PL is still an open problem in optimization.


In this paper, we show that the discrete-time HB has an accelerated linear rate under PL when the non-convexity is averaged-out.
To explain this notion further, we need to introduce a quantity called average Hessian.
\begin{definition} 
Let $w_* \in \reals^{d}$ be a global minimizer of a twice differentiable function $f(w): \reals^{d} \rightarrow \reals$. We define the \textbf{average Hessian} of $f(\cdot)$ at $w \in \reals^{d}$ to be:
\begin{equation} \label{avgH}%
H_f(w) :=  \int_0^1 \nabla^2 f (\theta w + (1-\theta) w_* ) d \theta.
\end{equation}
We say that \textbf{the non-convexity of $f(\cdot)$ between $w$ and $w_{*}$ is averaged-out}
with parameter $\lambda_{*}$ when the smallest eigenvalue of
the average Hessian satisfies $\lambda_{{\min}}(H_f(w)) \geq \lambda_{*}> 0$. 
For brevity, we sometimes say $f(\cdot)$ satisfies \av{\lambda_*} at $w$.
\end{definition}
From \eqref{avgH}, it should be clear that the average Hessian at $w$ is the average Hessian between the line segment that connects $w$ and a global optimal point $w_{*}$.

To get a flavor of our main result, we give an informal theoretical statement as follows.
\begin{mdframed}
\begin{theorem} (Informal version of Theorem~\ref{thm:meta})
Denote $w_{*}$ a global minimizer of $f(\cdot)$.
Suppose $f(\cdot)$ is twice differentiable, satisfies $\mu$-PL, has $L$-Lipschitz gradient and \textbf{diagonal} Hessian.
Apply HB to solve $\min_{w} f(w)$.
Assume that \av{\lambda_{\min}(H_f(w_t))} holds for all $t$.
Then, there exists a time $t_0 = \tilde{\Theta}(\frac{L}{\mu})$ such that for all $T > t_{0}$, the iterate $w_{T}$ of HB satisfies
\[
\| w_{T} - w_* \| = O\left( \prod_{t=t_0}^{T-1} \left(  1 - \Theta\left(\frac{1}{ \sqrt{\kappa_{t}}} \right)  \right) \right) \| w_{t_0} - w_* \|,
\]
where $\kappa_{t}:= \frac{L}{\lambda_{\min}(H_f(w_t)) }$ is the condition number of the average Hessian at $w_{t}$.
\end{theorem}
\end{mdframed}

To summarize, our contributions in this paper include:
\begin{itemize}
\item We develop an analysis of showing acceleration via HB beyond the quadratic problems, which can be of independent interest.
Our theoretical results reveal a way to check if incorporating the dynamic of momentum helps converge faster than GD under PL.
Our theoretical results will also show a benefit of adaptively setting the momentum parameter. 
\item We show that \av{\lambda_*} implies PL. We provide some concrete examples for which \av{\lambda_*} holds.
\end{itemize}

\section{Preliminaries}

\noindent
\textbf{Notations:}
We use $\beta_{t} I_{d} \in \reals^{{d \times d}}$ to denote a diagonal matrix on which each of its diagonal elements is $\beta_{t}$. We will have $\| M \|_2$ represent the spectral norm of a matrix $M$, i.e., $\|M \|_2 = \sqrt{ \lambda_{\max}( MM^*)}$, where $\lambda_{{\max}}(\cdot)$ is the largest eigenvalue of the underlying matrix, and $M^*$ is the conjugate transpose of the matrix $M$. We use $\| v \|_{2}$ as the $l_{2}$ norm of a vector $v$.
The notation $\text{Diag}( \cdots )$ in this paper represents a block-diagonal matrix that puts its arguments on the main diagonal. 

\noindent
\textbf{Definitions:}
A function $f(\cdot)$ is $L$-smooth 
if it has Lipschitz gradients, i.e., $\| \nabla f(x)- \nabla f(y) \|_{*} \leq L \| x - y\| $ for any $x,y$.
In this paper, we consider $l_{2}$ norm so that the dual norm is itself.
A function $f(\cdot)$ is $\nu$-strongly convex if $f(x) \geq f(y) + \langle \nabla f(y), x - y \rangle + \frac{\nu}{2} \| x - y \|^2$
for any $x,y$.
When $f(\cdot)$ is twice differentiable, the $\nu$-strong convexity is equivalent to the condition that the smallest eigenvalue of Hessian satisfies $\lambda_{{\min}}( \nabla^2 f(\cdot)) \geq \nu >0$.
A function $f(\cdot)$ is $\mu$-Polyak-\L{}ojasiewicz condition ($\mu$-PL) at $w$, if 
$\| \nabla f(w) \|^2 \geq 2 \mu ( f(w) - \min_w f(w) ).$
Observe that if $\mu$-PL holds for all the points in the domain of $f(\cdot)$, then 
all the stationary points are global minimizers of $f(\cdot)$. However, different from strongly convex functions, a PL-function does not necessarily have a unique global minimizer \citep{AGV21}.
An example is $f(w) = \frac{1}{2} w^{\top} A w$, where $A = A^{\top} \succeq 0$, which has $\arg\min f(w)=\mathrm{Ker}(A)$.
It is known that $\mu$-strongly convex implies $\mu$-PL \citep{KNS16}.

\subsection{Prior analysis of acceleration via HB}

Algorithm~\ref{alg:HB2} shows Heavy Ball, which has another 
equivalent version presented in the appendix (Algorithm~\ref{alg:HB1}).
 Our presentations of HB cover the case of HB with a constant momentum parameter, i.e., when $\beta_{t} = \beta$,
 which was studied in all the aforementioned works of HB.

The HB dynamic can be written as
\begin{equation} \label{eq1}
\begin{split}
w_{t+1}
& =  w_t  - \eta ( \nabla f(w_t) - \nabla f(w_*) ) + \beta_t  (w_t - w_{t-1} ) 
\\  &= w_t  - \eta \underbrace{ \left( \int_0^1 \nabla^2 f (\theta w_t + (1-\theta) w_* ) d \theta \right) }_{= H_f(w_t)} \left( w_t - w_*  \right) \\ & \quad + \beta_t ( w_t - w_{t-1} ),
\end{split}
\end{equation}
where in the first equality 
we denoted $w_{*}$ an optimal point and used $\nabla f(w_*) = 0$, and 
in the second equality, we used the fundamental theorem of calculus, i.e., $\nabla f(w_t) - \nabla f(w_*) = \left( \int_0^1 \nabla^2 f (\theta w_t + (1-\theta) w_* ) d \theta \right) \left( w_t - w_*  \right)  $.
Equation~\eqref{eq1} can be further rewritten as
\begin{equation} \label{dynamic}
\begin{split}
\begin{bmatrix}
w_{t+1} - w_* \\
w_{t} - w_* 
\end{bmatrix}
 &
=
\underbrace{
\begin{bmatrix}
I_{d} - \eta H_t + \beta_t I_{d} & - \beta_t  I_{d}   \\
I_{d} & 0_{d} 
\end{bmatrix}
}_{:= A_t}
\begin{bmatrix}
w_t - w_* \\
w_{t-1}- w_* 
\end{bmatrix},
\end{split}
\end{equation}
where we defined $H_{t}:= H_f(w_t)$ for brevity.

The most well-known acceleration result of HB is about applying HB with a constant $\beta$ for solving the strongly convex quadratic problems,
$ \min_{w \in \reals^d} \frac{1}{2} w^\top M w + b^\top w,$
where $M \in \reals^{{d \times d}}$ is a positive-definite matrix that has 
the largest eigenvalue
$\lambda_{\max}( M)= L$ and the smallest eigenvalue $\lambda_{\min}( M) =  \nu > 0$.
Let $H_{t} \gets M$ and $\beta_{t} \gets \beta$ in (\ref{dynamic})
and recursively expand the equation from time $t$ back to time $0$.
We have
\begin{equation}  \label{eq2}
\begin{bmatrix}
w_{t} - w_* \\
w_{t-1} - w_* 
\end{bmatrix}
=
A^t
\begin{bmatrix}
w_0 - w_* \\
w_{-1}- w_* 
\end{bmatrix}
,
\end{equation}
where  
$A:=\begin{bmatrix}
I_{d} - \eta M + \beta I_{d} & - \beta  I_{d}   \\
I_{d} & 0_{d}
\end{bmatrix} 
\in \reals^{2d \times 2d}$
and $A^t$ is its $t_{\mathrm{th}}$ matrix power.
To get the convergence rate from \eqref{eq2}, there are two approaches in the literature, which are described in details in the following.

\textbf{Prior approach 1 (bounding the spectral radius):} 
To upper-bound the convergence rate of the distance $\| w_t - w_* \|_2$ in \eqref{eq2}, it suffices to bound the spectral norm $\| A^t \|_{2}$.
Traditional analysis uses the spectral radius of $A$ to approximate the spectral norm $\| A^t \|_{2}$,
see e.g., \citep{P64,LRP16,R18,M19,OBP15}.
The spectral radius of a square (not necessarily symmetric)  matrix $A \in \reals^{{2d \times 2d}}$ is defined as $\rho(A) := \max_{i \in [2d]} | \lambda_i(A)|$, where $\lambda_i(\cdot)$ is the $i_{\mathrm{th}}$ (possibly complex) eigenvalue.
By choosing the step size $\eta$ and the momentum parameter $\beta$ appropriately,
the spectral radius of $A$ in \eqref{eq2} is $\rho(A)=1 - c \frac{1}{\sqrt{\kappa}}$ for some constant $c>0$, where $\kappa:= \frac{L}{\nu}$ is the condition number.
Then, by Gelfand's formula \citep{G41}, one has for any square matrix $B$,
\begin{equation} \label{gel}
\text{ (Gelfand's formula) } \quad \| B^k \|_2 = \left( \rho(B) + \epsilon_k \right)^k,
\end{equation}
for some sequence $\{ \epsilon_k\}$ that converges to $0$ when $k \rightarrow \infty$. Hence the spectral norm $\| A^t \|_2$ can be approximated by the spectral radius raised to $t$, i.e., $\left(\rho(A)\right)^{t}$, in the sense that
$\| A^t \|_2 = \left( \left(1 - c \frac{1}{\sqrt{\kappa}}\right) + \epsilon_t \right)^t$. This suggests that
asymptotically the convergence rate of HB is an accelerated linear rate $O\left(1 - \Theta\left(\frac{1}{\sqrt{\kappa}}\right)\right)$, which is better than that of GD, i.e., $O\left(1-\Theta\left(\frac{1}{\kappa}\right)\right)$.
The weakness of this approach is that the convergence rate is not quantifiable at a finite $t$, because one generally cannot control the convergence rate $\epsilon_{t}$ in the Gelfand's formula. On the other hand,
one might suggest bounding the spectral norm $\| A \|_2$, since $\| A^t \|_2 \leq \|A \|_{2}^{t}$.
In Appendix~\ref{app:HB}, we show that $1 \leq \| A \|_2 $ for any $\beta \in [0,1]$ and any step size $\eta \leq \frac{1}{L}$, which implies that a convergence rate of HB cannot be obtained via this way.

\textbf{Prior approach 2 (bounding the matrix-power-vector product):} 

Noting that controlling
the size of the matrix-power-vector product
$\left \|
A^t
\begin{bmatrix}
w_0 - w_* \\
w_{-1}- w_* 
\end{bmatrix}
\right\|_2
$ is sufficient for controlling the convergence rate of the distance in \eqref{eq2},
\citet{WLA21} show that by choosing the step size $\eta = \frac{c_{\eta}}{L}$ and the momentum parameter $\beta = \left(1 - c \frac{1}{ \sqrt{\kappa} } \right)^2$ for some constant $c_{{\eta}}$ and $c >0$ appropriately, 
one has 
$\textstyle
\left \|
A^t
\begin{bmatrix}
w_0 - w_* \\
w_{-1}- w_* 
\end{bmatrix}
\right\|_2
\leq 4 \sqrt{\kappa} \left( 1 - \frac{1}{2 \sqrt{\kappa}} \right)^t 
\| w_0 - w_* \|
,$
which is a non-asymptotic accelerated linear rate and hence avoids the concern of the first approach. 


\begin{table}[t]
\footnotesize
\begin{center}
\begin{tabular}{|c | c | c|} 
 \hline
 Class of problems & Convergence rate  & Ref.\\ 
 \hline\hline
Strongly convex quadratic  & $\left(1- \Theta \left( \frac{1}{ \sqrt{\kappa} } \right) \right)^T$ & [1] \\ \hline
$L$-smooth $\nu$-strongly convex & $\dagger$  $\left(1- \Theta \left( \frac{\nu}{ L } \right) \right)^T$  & [2]   \\ \hline
 $L$-smooth convex  & $O\left(\frac{L}{T}\right)$ & [3]  \\ \hline
$L$-smooth $\mu$-PL & $\left(1- \Theta \left( \frac{\mu}{ L } \right) \right)^T$
& [4]  \\  \hline
 Over-parametrized  NN& $\ddagger$  $\left(1- \Theta \left( \frac{1}{ \sqrt{\kappa_0} } \right) \right)^T$ & [5]  \\ \hline
\end{tabular}
\end{center}
\caption{
\footnotesize
Existing results of the discrete-time HB convergence in terms of the optimality gap $f(w_T) - \min_w f(w)$.
References: [1] Theorem 9 in \cite{P64} and Theorem 7 in \cite{WLA21} . [2] Theorem 3 in \cite{KDP18}. $\dagger$ It is also known that a local acceleration happens when the iterate is sufficiently close to the global optimal point $w_{*}$ under an additional assumption that $f(\cdot)$ is twice-differentiable, see e.g., Theorem 8 in \citep{WLA21}. [3] Theorem 2 in \citep{GFJ15} and Theorem 9 in \citep{WAL21}. 
[4] Theorem 3 in \citep{KDP18}. [5] Theorem 9-10 in \citep{WLA21}.
$\ddagger$ Here $\kappa_{0}$ is the condition number of the neural tangent kernel matrix of an over-parametrized neural network at initialization.  
 } \label{table:rate}
\end{table}
\section{Analysis}

The goal of this section is to develop an analysis that enables showing acceleration beyond the quadratic problems. But to get the ball rolling, we start by discussing some technical difficulties of reaching the goal using
the prior approaches.

\subsection{Why do prior approaches of showing acceleration fail for non-quadratic problems?}

Recall we have the HB dynamic for minimizing a general twice differentiable function shown on \eqref{dynamic}. Recursively expanding \eqref{dynamic}, one gets
\begin{equation}  \label{eq3}
\begin{bmatrix}
w_{T} - w_* \\
w_{T-1} - w_* 
\end{bmatrix}
=
\left( \prod_{t=0}^T
A_t
\right)
\begin{bmatrix}
w_0 - w_* \\
w_{-1}- w_* 
\end{bmatrix},
\end{equation}
where $\prod_{t=0}^T A_t :=  A_{T} A_{{T-1}} \cdots A_{1} A_{0}$.
Therefore, a way to control the distance between $w_{T}$ and $w_{*}$ is to control the size of the spectral norm $\| \prod_{t=0}^T  A_t \|_{2}$.
The spectral radius approach would compute the spectral radius $\rho(A_t)$
and then invoke the Gelfand's formula \eqref{gel} (with $k=1$) for each step $t$, i.e.,
$\| A_t^1 \|_2 = \left( \rho(A_t) + \epsilon_{t,1} \right)^1$, where $\epsilon_{{t,1}}$ is the approximation error.
Then, one might hope $\| \prod_{t=0}^T  A_t \|_{2}$ or $\prod_{t=0}^T \| A_t \|_{2}$ could be approximated by $\prod_{{t=0}}^{T} \rho(A_t)$ well.

The issue is that for a small $k$, e.g., $k=1$, $\epsilon_{t,k}$ in the Gelfand's formula can be large, which makes the spectral radius a poor estimate of the spectral norm. Indeed, it is easy to cook up an example based on the matrix $A_{t}$ in \eqref{dynamic}. Let $d=1$ (one-dimensional), $\eta H_{t}=0.1$, and $\beta=0.9$. We have
\begin{equation}
\textstyle
A_t = \begin{bmatrix} 1 -0.1 + 0.9 &   -0.9 \\ 1  & 0 \end{bmatrix}.
\end{equation}
For this example, the spectral radius $\rho(A_{t})=0.9487$ and the spectral norm $\|A_t\|_{2}= 2.21$. So the approximation error is $\epsilon_{t,1} =2.21 - 0.9487 \approxeq 1.25$, which implies that the product $\prod_{{t=0}}^{T} \rho(A_t)$ can not be guaranteed to approximate the spectral norm arbitrarily well.
This example also shows that the spectral norm $\| A_t \|_{2}$ 
is not useful for analyzing the convergence rate, since it is larger than $1$
(see Appendix~\ref{app:HB} for more discussions).
To show a convergence, 
we need to identity a quantity that is relevant to the convergence rate of the distance and has a non-trivial upper bound which is smaller than $1$.
One might consider bounding the spectral norm of the matrix product in \eqref{eq3} directly, i.e., bounding $\| \prod_{t=0}^T A_t \|_2$. The issue is that the matrix product $\prod_{t=0}^T A_t$ does not have a simple analytical form. 

Regarding the approach of \citet{WLA21}, we have mentioned its limitation
in the introduction when dealing with the case that the Hessian changes significantly. We expound this further in Appendix~\ref{app:HB}. 
To summarize, the limitations of the previous techniques cause obstacles to analyzing HB when the Hessian changes a lot during the update. A new analysis is needed.

\subsection{New analysis}

Before showing our approach, we need the following technical result from \citep{WLA21}. 
The lemma below shows that while 
 the non-symmetric matrix $A$ does not have an eigen-decomposition in reals,
 it has an decomposition in the complex field
when the momentum parameter is larger than a threshold.
Moreover, the spectral norm of the diagonal matrix is $\sqrt{\beta}$.
The full version of Lemma~\ref{lem:diagonal} and its proof is in Appendix~\ref{app:lem1}.

\begin{lemma}(short version)
\label{lem:diagonal} (Lemma 6 in \cite{WLA21})
Consider a matrix
\begin{equation}
A:=
\begin{bmatrix} 
(1 + \beta) I_d -  H  & - \beta I_d\\ 
I_d & 0
\end{bmatrix}
\in \reals^{2 d \times 2 d},
\end{equation}
where $H \in \reals^{{d \times d}}$ is a symmetric matrix and has eigenvalues $\lambda_{1} \geq \lambda_2 \geq \dots \geq \lambda_{i} \geq \dots \geq \lambda_{d}$.
Suppose $\beta$ satisfies $1\geq \beta > \left( 1 - \sqrt{\lambda_i}  \right)^2$
for all $i \in [d]$. 
Then, we have
$A$ is diagonalizable with respect to the complex field $\mathbb{C}$ in $\mathbb{C}^{2d \times 2d}$:
\begin{equation}
(\textbf{spectral decomposition}) \quad A = PDP^{-1}
\end{equation}
for some matrix
$P = \tilde{U}\tilde{P} Q$,
where
$\tilde{U}$ and $\tilde{P}$ are some orthogonal matrices,
$Q ={\rm Diag}(Q_1,\dots,Q_d) \in \mathbb{C}^{{2d \times 2d}}$ is a block diagonal matrix,
and $D \in \mathbb{C}^{{2d \times 2d}}$ is a diagonal matrix.
Furthermore, the diagonal matrix $D$ 
satisfies $\| D \|_{2} = \sqrt{\beta}$.
\end{lemma}

Now we are ready to describe our approach.
Recall the dynamic \eqref{dynamic} and the definition of the average Hessian $H_{t}$ defined in \eqref{avgH}.
Using Lemma~\ref{lem:diagonal} with $H \gets \eta H_{t}$, 
we know the matrix $A_{t}$ in the dynamic \eqref{dynamic} has a decomposition 
in the complex field,
$A_t = P_tD_tP_t^{-1}$, where $D_{t}$ is a diagonal matrix whose spectral norm is
$\| D_{{t}} \|_2 = \sqrt{\beta_t}$.
Denote $\xi_{t}:= w_t - w_*$.
One can recursively expand the dynamic \eqref{dynamic} from $T+1$ back to time $t_{0}<T+1$ as follows.
\begin{equation} \label{eqeq}
\begin{split}
 & \textstyle \begin{bmatrix} \xi_{T+1} \\ \xi_T \end{bmatrix}  
   = 
A_T A_{T-1} \cdots A_{t_0}  \begin{bmatrix} \xi_{t_0} \\ \xi_{t_0-1} \end{bmatrix}
\\ & \textstyle = 
\left( P_TD_TP_T^{-1} \right)
\left( P_{T-1}D_{T-1}P_{T-1}^{-1} \right)
\left( P_{T-2}D_{T-2}P_{T-2}^{-1} \right)
\\ & \textstyle
\quad 
\cdots
\left( P_{t_0}D_{t_0}P_{t_0}^{-1} \right)
\begin{bmatrix} \xi_{t_0} \\ \xi_{t_0-1} \end{bmatrix}, 
\\ &= 
P_T \underbrace{ \left( D_TP_T^{-1} 
P_{T-1} \right) }_{:= \Psi_T} \underbrace{ \left( D_{T-1}P_{T-1}^{-1} 
P_{T-2} \right) }_{:= \Psi_{T-1} } \cdots
\\ & \textstyle
\cdot \underbrace{ \left( D_{T-2}P_{T-2}^{-1} P_{T-3} \right) }_{:= \Psi_{T-2}}
\cdots
\underbrace{ \left(D_{t_0+1} P_{t_0+1}^{-1} P_{t_0} \right) }_{:= \Psi_{t0}} D_{t_0}P_{t_0}^{-1}
\begin{bmatrix} \xi_{t_0} \\ \xi_{t_0-1} \end{bmatrix}. 
\end{split}
\end{equation}
Hence, one has
\begin{equation} \label{eq:multi}
\begin{split}
\textstyle
\left \| \begin{bmatrix} \xi_{T+1} \\ \xi_T \end{bmatrix} \right\|_2 & \leq \|P_T\|_2\left(\prod_{t=t_0+1}^T \underbrace{ \|D_tP_t^{-1}P_{t-1}\|_2 }_{:=\|\Psi_t \|_2 } \right)\\ & \textstyle \qquad \qquad \times \|D_{t_0}P_{t_0}^{-1}\|_2 \left \|\begin{bmatrix} \xi_{t_0} \\ \xi_{t_0-1} \end{bmatrix} \right\|_2.
\end{split}
\end{equation}
From \eqref{eq:multi}, we see that the 
distance can be bounded as
$\|\xi_{T}\|=O(\prod_{t=t_0+1}^T \|\Psi_t \|_2) \| \xi_{t_0}\|$, where we define:
\begin{definition}
\[
(\textbf{instantaneous rate at $t$}) \quad \|\Psi_t \|_2 := \|D_tP_t^{-1}P_{t-1}\|_2.
\]
\end{definition}
The instantaneous rate is determined by the interplay between eigenvectors of the average Hessians $H_t$ and $H_{{t-1}}$ at the two consecutive time points.
By Lemma~\ref{lem:diagonal}, we know
\begin{equation}
\|\Psi_t \|_2 := \|D_tP_t^{-1}P_{t-1}\|_2 \leq 
\sqrt{\beta_t} \|P_t^{-1}P_{t-1}\|_{2}.
\end{equation}
Therefore, to upper-bound the instantaneous rate, it suffices to upper-bound
$\|P_t^{-1}P_{t-1}\|_{2}$, which arises from the change of the eigen-space of the average Hessian. 
Our key finding is that $\|P_t^{-1}P_{t-1}\|_{2}$ admits a simple closed-form expression, which can be used to derive an upper bound of the instantaneous rate. 
Based on the decomposition from Lemma~\ref{lem:diagonal}, 
we have $P_t = \tilde{U}_t \tilde{P}_t Q_t$,
where $\tilde{U}_t$ and $\tilde{P}_t$ are some orthogonal matrices.
Then, $\| P_t^{-1}P_{t-1} \|_2 = \| Q_t^{-1} M_t Q_{t-1} \|_2$, where $M_t = (\tilde{P}_t)^{-1} (\tilde{U}_t)^{-1} \tilde{U}_{t-1}\tilde{P}_{t-1}$ is an orthogonal matrix.

In the following lemma, 
we denote the eigenvalues of the average Hessian $H_{t}$ in the dynamic \eqref{dynamic} as $\lambda_{t,1}\geq \lambda_{t,2} \geq \dots \geq \lambda_{{t},d}$.
The lemma requires the following \emph{co-diagonalization} to hold:
\begin{equation} 
\spadesuit: ~~
\| Q_t^{-1} M_t Q_{t-1} \|_2  \leq \| Q_t^{-1} Q_{t-1} \|_2
\end{equation}
which is guaranteed when the dimension $d$ of the optimization variable $w$ is $1$, or more broadly, when the Hessian is diagonal, as in these cases, the matrix $M_t$ is the identity matrix for all $t$.
We elaborate on why the condition holds for these examples in Appendix~\ref{app:when}.

\begin{lemma} (short version) \label{vb}
Assume we have $1\geq \beta_{t} > (1-\sqrt{\eta \lambda_{t,i}})^2$ for all $i \in [d]$ and that the co-diagonalization condition $\spadesuit$ holds.
Then, the {\bf instantaneous rate} $\| \Psi_t\|_2$ at $t$ satisfies:

(I) If $\lambda_{t,i} \geq \lambda_{t-1,i}$, then
\begin{equation} \label{vb1}
\begin{split}
& \|\Psi_t\|_2
 \leq \sqrt{\beta_t} \times
 \text{ extra factor 1},
\end{split}
\end{equation}
where the extra factor 1 is
\begin{equation} \label{vb3}
\underbrace{ \left( \max_{i \in [d]} \sqrt{1+ \frac{\eta \lambda_{t,i} -\eta \lambda_{t-1,i}}{ (1+\sqrt{\beta_t})^2 - \eta \lambda_{t,i}  } }
 + \mathbbm{1}\{\beta_t \neq \beta_{t-1}\} \phi_{t,i}
\right) }_{\text{extra factor 1} }.
\end{equation}

(II) If $\lambda_{t-1,i} \geq \lambda_{t,i}$, then
\begin{equation} \label{vb2}
\begin{split}
\|\Psi_t\|_2
& \leq \sqrt{\beta_t} 
\times \text{ extra factor 2},
\end{split}
\end{equation}
where the extra factor 2 is
\begin{equation} \label{vb4}
\underbrace{ \left( \max_{i \in [d]}
\sqrt{1+ \frac{\eta \lambda_{t-1,i} - \eta \lambda_{t,i}}{ \eta \lambda_{t,i}  - (1-\sqrt{\beta_t })^2  } } 
 + \mathbbm{1}\{\beta_t \neq \beta_{t-1}\} \phi_{t,i}
\right) }_{\text{extra factor 2} }.
\end{equation}
The term $\phi_{{t,i}}$ in \eqref{vb3} and \eqref{vb4}
has an upper-bound which is a function of $| \beta_{t-1} - \beta_t|$, defined in the full version.
\end{lemma}
The full version of Lemma~\ref{vb} and its proof is in Appendix~\ref{app:lem2}. 
Lemma~\ref{vb} hints at what the optimal momentum parameter is. It says that as long as $\beta_{t}$ is larger than the threshold $\max_{i \in [d]} (1-\sqrt{\eta \lambda_{t,i}})^2 $, 
the instantaneous rate is the square root of the parameter
$\sqrt{\beta_t}$ modulo an extra factor.
So the momentum parameter should not be set far larger than the threshold.
The extra factors in the lemma are due to bounding $\|P_t^{-1}P_{t-1}\|_{2}$, and
we would like the extra factors to be as small as possible to obtain an accelerated linear rate.
To see this, consider setting $\eta = \Theta\left( \frac{1}{L} \right)$. Then,
we will have $\sqrt{ \beta_t } =  1 - \Theta \left( \frac{1}{\sqrt{\kappa_t} }   \right) $, where $\kappa_{t}:= \frac{L}{\lambda_{\min}(H_t)}$ is the condition number of $H_{t}$. That is, the instantaneous rate $\| \Psi_t\|_2$ will be
\begin{equation} \label{eq:int}
\textstyle
\| \Psi_t\|_2 =
\left( 1 - \Theta \left( \frac{1}{\sqrt{\kappa_t} } \right)  \right)
\times \text{extra factor 1 or 2}.
\end{equation}
Hence if the extra factors decay fast enough over time so that they are approximately $1$ after some number of iterations $t_{0}$,
then for all $t \geq t_{0}$ the instantaneous rate $\| \Psi_t\|_{2}$ will be an accelerated linear rate
$\| \Psi_t\|_2 = 
 1 - \Theta \left( \frac{1}{\sqrt{\kappa_t} } \right)$.

Observe that the extra factors \eqref{vb3} and \eqref{vb4} have the term
$| \lambda_{t,i} - \lambda_{t-1,i} |$, which is resulted from the change of the average Hessian.  
In the rest of this subsection, we discuss how to bound this term, as it will help  upper-bound the extra factors and consequently the instantaneous rate $\|\Psi_t\|_{2}$.
We start by connecting the difference $| \lambda_{t,i} - \lambda_{t-1,i} |$
and the distance between $w_t$ and $w_{{t-1}}$.

\begin{lemma} \label{lem:lambdadiff}
Suppose that the Hessian of $f(\cdot)$ is $L_{H}$-Lipschitz,
i.e.,
$\| \nabla^2 f(x) - \nabla^2 f(y) \|_2 \leq  L_H \| x - y \|_2,$
for any pair of $x,y$.
Then, 
$| \lambda_{t,i} - \lambda_{{t-1},i} | \leq L_H \| w_t - w_{t-1}\|_2.$
\end{lemma}

The proof of Lemma~\ref{lem:lambdadiff} is in Appendix~\ref{app:lem:lambdadiff}.
Lemma~\ref{lem:lambdadiff} suggests that it suffices to bound the distance 
$\| w_t - w_{t-1}\|_2$ in order to control the extra factors.
The strategy now is to show that the distance $\| w_t - w_{t-1}\|_2$ decays exponentially fast so that the instantaneous rate \eqref{eq:int} becomes
$\| \Psi_t\|_2 =  1 - \Theta \left( \frac{1}{\sqrt{\kappa_t} } \right)$ 
after some point $t_{0}$.
\citet{KDP18} show that for minimizing functions satisfying $\mu$-PL and $L$-smooth via HB, the distance $\| w_t - w_{t-1}\|_2$ decays at an linear rate $1 - \Theta\left( \frac{\mu}{L} \right)$ when the momentum parameter is a constant, i.e., $\beta_{t} = \beta $, and satisfies a constraint.
Specifically,
\citet{KDP18} construct the following Lyapunov function:
\begin{equation} \label{lyp0}
\textstyle
\bar{V}_t := f(w_t) - \min_w f(w) 
+ \frac{L(1-c_{\eta})}{2 c_{\eta}} \| w_t - w_{t-1} \|^2,
\end{equation}
where $\eta = \frac{c_{\eta}}{L}$ for some $c_{{\eta}} \in (0,1)$ is the step size of HB and $L$ is the smoothness constant.
\citet{KDP18} show that the Lyapunov function decays at an linear rate,
i.e., $\bar{V}_{t} = \left( 1 - \Theta\left( \frac{\mu}{L} \right) \right) \bar{V}_{t-1}$, 
when the momentum parameter satisfies $\beta \leq
\sqrt{ \left( 1 - \frac{c_{\eta} \tilde{c} \mu}{L} \right)  
\left( 1-c_{\eta} \right) }$
for some constant $\tilde{c} \in (0,1]$,
which implies that the distance between the iterates at two consecutive time points shrinks linearly towards $0$ in the sense that $\| w_t - w_{t-1}\|^2 = \left( 1 - \Theta\left( \frac{\mu}{L} \right) \right)^t  \bar{V}_0$.
Here we show a similar result under a less restricted constraint on the momentum parameter. Moreover, our result allows the value of the momentum parameter to change during the update.  We construct the following Lyapunov function:
\begin{equation} \label{lyp}
\textstyle
V_t := f(w_t) - \min_w f(w) 
+ \theta \| w_t - w_{t-1} \|^2,
\end{equation}
where $\theta>0$ is a constant to be determined in Theorem~\ref{thm:PL}.


\begin{theorem} \label{thm:PL}
Let
$\theta = 2 \left( \frac{L}{4} \left( 1 + \frac{1}{c_{\eta}}\right) - c_{\mu} \mu    \right) > 0$ for any $c_{{\eta}} \in (0,1]$ 
and any $c_{{\mu}} \in (0,\frac{1}{4}]$.
Set the step size $\eta = \frac{c_{\eta} }{L}$ and set the momentum parameter $\beta_{t}$ so that for all~$t$,
$\beta_t \leq \sqrt{ 
\left(1 - \tilde{c} c_{\eta}^2 \frac{\mu}{L}  \right) 
\left( 1 - \frac{c_{\mu} \mu }{  \frac{L}{4} \left( 1 + \frac{1}{c_{\eta}}\right) + \frac{\theta}{2} }
\right)}$
for some constant $\tilde{c} \in (0,1]$.
Then, HB has
\begin{equation} \label{eq:V}
V_t
 \leq \left( 1 - \tilde{c} c_{\eta}^2 \frac{ \mu }{L }  \right)^t V_0
 = \left( 1 - \Theta\left( \frac{ \mu }{L} \right)  \right)^t V_0,
\end{equation}
where the Lyapunov function $V_{t}$ is defined on \eqref{lyp}.
\end{theorem}

Theorem~\ref{thm:PL} shows that the optimality gap converges at a linear rate $1 - \Theta\left( \frac{\mu}{L} \right)$.
It also shows that the distance between the two consecutive iterates shrinks fast.
Our main theorem in the next subsection will use Theorem~\ref{thm:PL} to upper-bound the extra factor in Lemma~\ref{vb}.
The proof of Theorem~\ref{thm:PL} is deferred to Appendix~\ref{app:thm:PL}.

\subsection{Main theorem}

Our main theorem (Theorem~\ref{thm:meta}) will assume the following: \\
\noindent
\textbf{Assumption $\clubsuit$:}
\textit{The function $f(\cdot)$ satisfies $\mu$-PL, is twice differentiable, $L$-smooth, and has $L_{H}$-Lipschitz Hessian.
}


\begin{theorem} \label{thm:meta}
Suppose assumption $\clubsuit$ holds
and that \avg{\lambda_{\min}(H_t)}{w_*} and co-diagonalization $\spadesuit$ hold for all $t$.
Let
$\theta := 2 \left( \frac{L}{4} \left( 1 + \frac{1}{c_{\eta}}\right) - c_{\mu} \mu    \right) > 0$, where $c_{{\mu}} \in (0,\frac{1}{4}]$.
Set the step size $\eta = \frac{c_{\eta} }{L}$
for any constant $c_{\eta} \in (0,1]$
and set the momentum parameter 
$\beta_t = \left(1- c \sqrt{  \eta \lambda_{\min}(H_t) } \right)^2 $ for some $c \in (0,1)$
satisfying 
$\beta_t \leq
\sqrt{
\left(1 - \tilde{c} c_{\eta}^2 \frac{\mu}{L}  \right) 
\left( 1 - \frac{c_{\mu} \mu }{  \frac{L}{4} \left( 1 + \frac{1}{c_{\eta}}\right) + \frac{\theta}{2} }
\right) }$
for some constant $\tilde{c} \in (0,1]$.
Then, for all $t$, the iterate $w_{t}$ of HB satisfies \eqref{eq:V}, i.e.,
the Lyapunov function $V_{t}$ decays linearly for all $t$.
Furthermore,
there exists a time $t_{0} = \tilde{\Theta} \left( \frac{L}{\mu}
 \right)$
such that for all $t \geq t_{0}$, the instantaneous rate $\| \Psi_t \|_2$ at $t$ is
\begin{equation} \label{eq:psi}
\| \Psi_t \|_2 = 1 - \frac{c \sqrt{c_{\eta}}}{2} \frac{1}{\sqrt{\kappa_t}}
= 1 - \Theta\left( \frac{1}{\sqrt{\kappa_t} }  \right), 
\end{equation}
where $\kappa_t:= \frac{L}{\lambda_{\min}(H_t)}$.
Consequently, 
\[
\| w_{T+1} - w_* \| = O\left( \prod_{t=t_0}^{T} \left(  1 - \Theta\left(\frac{1}{ \sqrt{\kappa_{t}}} \right)  \right) \right) \| w_{t_0} - w_* \|.
\]
\end{theorem}

The proof of Theorem~\ref{thm:meta} is given in Appendix~\ref{app:thm:meta}, which is built on the results developed in the previous subsection.
Theorem~\ref{thm:meta} shows that HB has a provable advantage over GD even when the optimization landscape is non-convex, as long as the average Hessian towards $w_{*}$ is positive definite. 
In the next section we will give some examples where the averaged-out condition holds.
But at this moment, let us elaborate on our theoretical statement of acceleration by
making a detailed comparison between GD and HB.
Observe that the dynamic of the distance due to GD can be written as
\begin{equation} \label{GDd}
\begin{split}
\textstyle w_{t+1} - w_*  & \textstyle =  w_t - \eta \nabla f(w_t) - w_*
\\ & \textstyle
=  \left( I_d - \eta H_t \right) (w_t - w_*),
\end{split}
\end{equation}
where $H_{t}$ is the average Hessian towards $w_{*}$ defined in \eqref{avgH}
and we used the fundamental theorem of calculus in the second equality as \eqref{eq1}.
Then,
choosing the step size $\eta = \frac{c_{\eta}}{L}$ for some $c_{{\eta}}$,
we have from \eqref{GDd},
\begin{equation} \label{GDd2}
\textstyle \| w_{t+1} - w_* \|_2 = \left( 1  - \Theta\left( \frac{\lambda_{\min}(H_t) }{L} \right) \right) \| w_t - w_* \|_2.
\end{equation}
The inequality \eqref{GDd2} shows that the instantaneous rate of GD at $t$ is $1  - \Theta\left( \frac{\lambda_{\min}(H_t) }{L} \right)$.
On the other hand, GD under PL is known \citep{KNS17} to have
\begin{equation} \label{GDd3}
\textstyle f(w_t) - f(w_*) = \left( 1 - \Theta\left( \frac{\mu}{L} \right)   \right)^t \left( f(w_0) - f(w_*)  \right).
\end{equation}
Comparing \eqref{GDd2} and \eqref{GDd3} of GD to \eqref{eq:V} and \eqref{eq:psi} of HB in Theorem~\ref{thm:meta},
we can conclude that both GD and HB are guaranteed to converge
at the rate $1 - \Theta\left( \frac{\mu}{L} \right) $ for all $t$.
On the other hand, after certain number of iterations $t_{0}=\tilde{\Theta} \left( \frac{L}{\mu} \right)$, 
where
the notation $\tilde{\Theta}(\cdot)$ hides a logarithmic factor 
\footnote{The logarithmic factor is shown on \eqref{103} in Appendix~\ref{app:thm:meta}.}
that depends on $L_{H}$, $V_{0}$, $\theta$, and other parameters, 
HB has the instantaneous rate
$ 1  - \Theta\left( \sqrt{ \frac{\lambda_{\min}(H_t) }{L} } \right)  =  1 - ~ \Theta\left( \frac{1}{ \sqrt{\kappa_t} } \right)  $,
which is better than 
$ 1  - \Theta\left(  \frac{\lambda_{\min}(H_t) }{L}  \right) = 
1 - \Theta\left( \frac{1}{\kappa_t} \right) $ 
of GD.

The reader would notice that for the acceleration result to be meaningful, we need
$ 1  - \Theta\left( \sqrt{ \frac{\lambda_{\min}(H_t) }{L} } \right)  <  1 - \Theta\left( \frac{\mu}{L} \right)   $.
Hence it is desired to have 
$\lambda_{\min}(H_t) \geq \lambda_{*} > 0 , \forall t$ and that
$\lambda_* = \Theta\left( \mu \right)$,
which is actually problem-dependent. But we will discuss this in the next subsection and will
show that all the examples given in the last part of this paper have the desired property.

Let us emphasize that the parameters $c$ and $c_{\eta}$, which appear in the instantaneous rate on \eqref{eq:psi}, can indeed be any \emph{independent universal constant} in $(0,1)$. 
Since Theorem 3 indicates 
setting the momentum parameter 
$\beta_t = (1- c \sqrt{c_{\eta}} \sqrt{ \frac{\lambda_{\min}(H_t) }{L} } )^2 $ for any $c \in (0,1)$ and $c_{{\eta}} \in (0,1]$, 
we need to check if the choice of $\beta_{t}$ satisfies the upper-bound constraint, i.e., check if $
\left(1- c \sqrt{c_{\eta}} \sqrt{ \frac{\lambda_{\min}(H_t) }{L} } \right)^2
\leq \sqrt{ 
\left(1 - \tilde{c} c_{\eta}^2 \frac{ \mu }{L } \right)
\left( 1 - \frac{c_{\mu} \mu }{  \frac{L}{4} ( 1 + \frac{1}{c_{\eta}}) + \frac{\theta}{2} }
\right)}$. To show this, it suffices to show 
\begin{equation} \label{22}
\textstyle \left(1- c \sqrt{c_{\eta}} \frac{\lambda_{\min}(H_t) }{L}  \right)^2
\leq  
\left(1 -  \tilde{c} c_{\eta}^2  \frac{\mu}{L}  \right) 
\left( 1 - \frac{4 c_{\mu} \mu }{ L } \right).
\end{equation}
Let $\lambda_{*}$ be the constant such that $\lambda_{\min}(H_t) \geq \lambda_{*}> 0, \forall t$. 
Then, \eqref{22} holds if $1- c \sqrt{c_{\eta}} \frac{\lambda_*}{L}  \leq 1 -\tilde{c} c_{\eta}^2 \frac{\mu}{L}   $
and $1- c \sqrt{c_{\eta}} \frac{\lambda_*}{L} \leq  1 - 4 c_{\mu}  \frac{\mu }{ L } $.
Define the ratio $r= \frac{\lambda_*}{\mu}$.
Then, the constraint is equivalent to 
\begin{equation} \label{con}
\textstyle
\tilde{c} \leq c c_{\eta}^{-3/2} r \text{ and } c_{{\mu}} \leq\frac{c \sqrt{c_{\eta}} r}{ 4}.
\end{equation}
The condition on \eqref{con} can be easily met.
For example, we can choose $c=0.9$ and $c_{{\eta}}=1.0$ so that they are universal constants and do not depend on other problem parameters. 
Under this choice, the condition \eqref{con} becomes
$\tilde{c} \leq 0.9 r$ and $c_{{\mu}} \leq 0.225 r$, which 
we can easily satisfy.
It is noted that the parameter $\tilde{c}$ and $c_{{\mu}}$ \emph{only} affects the number of iterations spent in the burn-in stage. 
In summary, $c \in (0,1)$ and $c_{{\eta}} \in (0,1]$ in the theorem can be independent universal constants, while $\tilde{c}$ and $c_{{\mu}}$ in Theorem~\ref{thm:PL} for the burn-in stage might depend on the ratio $r= \frac{\lambda_{*} }{ \mu}$ under the mild condition $\eqref{con}$.


The presence of the \emph{burn-in} stage before exhibiting an accelerated linear rate is standard for HB, c.f., \citet{P64,LRP16}.
\citet{KDP18} provides a concrete example that shows HB could have a peak effect during the initial stage and hence prevents from obtaining an accelerated linear rate in the first few iterations. Indeed the acceleration results (of other problems) in \citet{WLA21} all have the form
$\| w_t - w_* \| \leq C \sqrt{\kappa} \left( 1 - \Theta\left( \frac{1}{\sqrt{\kappa}} \right)\right)^t \| w_0 - w_*\|$, for some constant $C \geq 1$. The upper-bound of the convergence rate is smaller than the initial distance only after certain number of iterations.
We also remark that if one wants to relate the distance $\| w_{t_0} - w_*\|$ in Theorem~\ref{thm:meta} to the initial distance $\| w_0 - w_*\|$, then one can use the fact the Lyapunov function $V_{t}$ decays exponentially 
and the triangular inequality to show that $\| w_{t_0} - w_*\|$ is within a constant multiple of 
$\| w_0 - w_* \|$ in the worst case. Furthermore, as $\mu$-PL implies
the quadratic growth condition with the same constant, i.e., $\frac{\mu}{2} \| w_{t_0} -w_* \|^2 \leq f(w_{t_0}) - f(w_*)$ (see e.g., \citep{KNS16}), one can also derive a bound on $\| w_{t_0} - w_*\|$ by using the quadratic growth condition and the fact that $V_t$ 
decays exponentially.

Finally we remark that Theorem~\ref{thm:meta} also finds that
the momentum parameter should be set adaptively. It basically says that if the momentum parameter is $\beta_{t}=\left(1 - \Theta\left( \frac{1}{\sqrt{\kappa_t}}\right) \right)^2$, then the instantaneous rate is 
$1 -~ \Theta\left( \frac{1}{\sqrt{\kappa_t}}\right) $ after some number of iterations,
assuming all the constraints are satisfied. 
Let $\lambda_{{*}} = \min_{t} \lambda_{\min}(H_t)$ and $\kappa:= L / \lambda_{*}$.
It implies that the instantaneous rate can be much smaller than $1 - \Theta\left( \frac{1}{\sqrt{{\kappa}}}\right) $ since $\kappa\geq \kappa_{t}$.
In other words, there is an advantage when the momentum parameter is set adaptively and appropriately over the iterations.
Nevertheless, for the completeness, we present the following corollary for the case when a constant value of  $\beta$ is used.
\begin{corollary} (\textbf{Constant $\beta$}) \label{cor}
Suppose assumption $\clubsuit$ holds.
Denote $\kappa:= \frac{L}{\lambda_*}$ for some constant $\lambda_{*}>0$. 
Let $\theta$ be that in Theorem~\ref{thm:meta}.
Set the step size $\eta = \frac{c_{\eta} }{L}$
for any constant $c_{\eta} \in (0,1]$
and set the momentum parameter 
$\beta = \left(1- \frac{ c \sqrt{c_{\eta} } }{ \sqrt{ \kappa} } \right)^2  $ for some $c \in (0,1)$
satisfying 
$\beta \leq \sqrt{ \left( 1 -  \tilde{c} c_{\eta}^2 \frac{ \mu}{L} \right)  
\left( 1- 4 c_{\mu} \frac{ \mu }{L} \right) }$
for some constant $\tilde{c} \in (0,1]$ and $c_{\mu} \in (0,\frac{1}{4}]$.
Then, \eqref{eq:V} holds for all $t$.
Furthermore, 
if HB satisfies \avg{\lambda_*}{w_*} and co-diagonalization $\spadesuit$ for all $t$,
then 
there exists a time $t_{0} = \tilde{\Theta} \left( \frac{L}{\mu} \right)$
such that for all $t \geq t_{0}$, the instantaneous rate $\| \Psi_t \|_2$ at $t$ is
$
\| \Psi_t \|_2 = 1 - \frac{c \sqrt{c_{\eta}}}{2} \frac{1}{\sqrt{\kappa}} = 1 - \Theta\left( \frac{1}{\sqrt{\kappa} }  \right)$. 
Consequently, 
$\| w_{T+1} - w_* \| = 
O\left( \left(  1 - \Theta\left(\frac{1}{ \sqrt{\kappa} } \right)\right)^{T-t_0}  \right) \| w_{t_0} - w_* \|.$
\end{corollary}

We conclude this subsection by pointing out that our presentation of the theoretical results aims at showing acceleration of HB can be achieved under various choices of the parameters. One can, for example, simply set the step size as $\eta = \frac{1}{L}$
and set the momentum parameter as
$\beta_t = \left(1- 0.9 \frac{ 1}{ \sqrt{\kappa_t} } \right)^2 $
or a constant $\beta = \left(1- 0.9 \frac{1 }{ \sqrt{\kappa }} \right)^2$
in HB (which corresponds to $c=0.9$ 
\footnote{The parameter $c$ cannot be set to $1$, because our proof in Appendix~\ref{app:thm:meta} shows that the \emph{logarithmic} factor in $t_{0}= \tilde{\Theta}\left(\frac{L}{\mu} \right)$ depends on how close $c$ to $1$ is.
More precisely, as $c$ approaches $1$, the denominator terms in the extra factors shown in Lemma~\ref{vb} could become zero.
} 
and $c_{{\eta}}=1$).
Under this choice of the parameters, 
HB exhibits acceleration after $t_{0}= \tilde{\Theta}\left(\frac{L}{\mu} \right)$ as guaranteed by
Theorem~\ref{thm:meta} or Corollary~\ref{cor}.

\subsection{\ao implies PL}

In the subsection, we show that if $\ao$ holds at a point $w$ then PL also holds at $w$.

\begin{theorem} \label{lem:AOPL}
Assume $f(\cdot)$ is $L$-smooth and twice differentiable. 
Consider a point $w \in \reals^{d}$. Suppose there exists a global optimal point $w_{*}$ of $f(\cdot)$
such that the condition \avg{\lambda_*}{w_*} holds at $w$.
Then, $f(\cdot)$ satisfies $\mu$-PL with parameter $\mu = \frac{\lambda_*^2}{L}$ at $w$
(, \textbf{which can be improved to $\mu=\lambda_*$ with a refined analysis}
\footnote{In a personal communication, Sinho Chewi pointed out that
\av{\lambda_*} actually implies $\lambda_*$-PL, which is an improved result compared to Theorem~\ref{lem:AOPL}. Please see Section~\ref{app:sinho} for the details.}).

\end{theorem}

\begin{proof}
If \avg{\lambda_*}{w_*} holds at $w$, then
\begin{equation} \label{eq111}
\begin{aligned} 
\textstyle \frac{ \| \nabla f(w) \| }{ \| w - w_* \|   }
&\textstyle  = \frac{ \| \left( \int_0^1 \nabla^2 f( (1-\theta) w + \theta w_*) d\theta  \right) (w-w_*) \| }{ \| w - w_*\|  }
\textstyle  \geq \lambda_*,
\end{aligned}
\end{equation}
where the equality uses $\nabla f(w_*)=0$ and the fundamental theorem of calculus,
and the inequality is due to $\| H_f(w) v \|_{2} \geq \lambda_{{\min}}(H_f(w)) \| v \|_2$ for any $v \in \reals^d$ and the assumption that $\ao{\lambda_*}$ holds.
Combining this and the smoothness, we have
\begin{equation}
\begin{split}
\textstyle  f(w) 
 & \textstyle  \leq f(w_*) + \frac{L}{2} \| w- w_* \|^2 \leq f(w_*) + \frac{L}{2 \lambda_*^2} \| \nabla f(w)\|^2,
\end{split}
\end{equation}
where the last inequality uses \eqref{eq111}. This shows
$f(\cdot)$ satisfies $\mu$-PL with parameter $\mu = \frac{\lambda_*^2}{L}$ at $w$.
\end{proof}
We remark that in the lemma $\lambda_{*} \geq \frac{\lambda_*^2}{L} =\mu$ since $L \geq \lambda_{*}$.

\section{Examples of $\ao$}

\subsection{Example 1}

If a twice differentiable function $f(\cdot)$ is $\nu$-strongly convex, then evidently the smallest eigenvalue of the average Hessian is at least $\nu$. Thus, \av{\lambda_*} holds with $\lambda_{*}= \nu$. 


\begin{figure}[t]
     \includegraphics[width=0.23\textwidth]{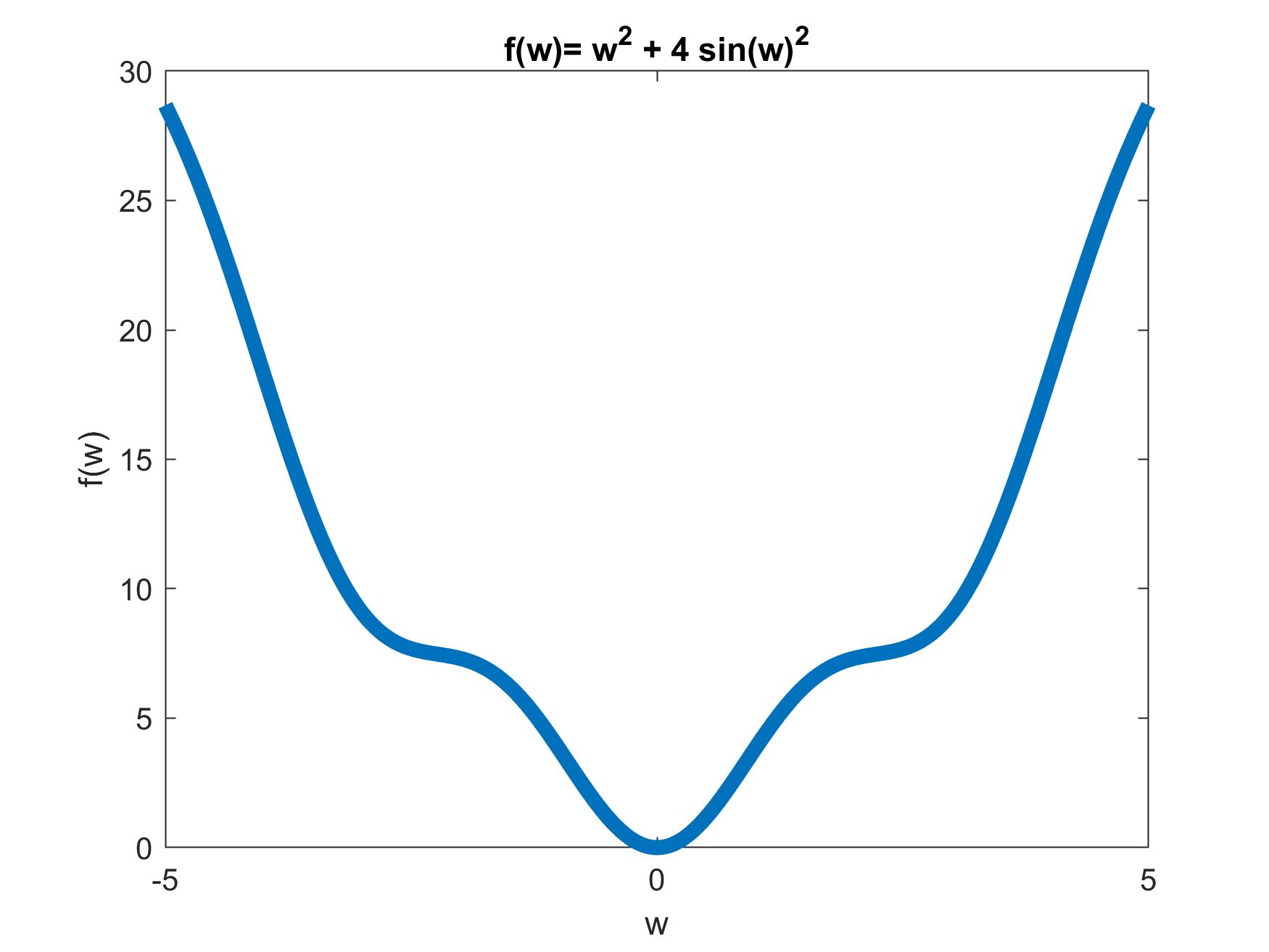}
     \includegraphics[width=0.23\textwidth]{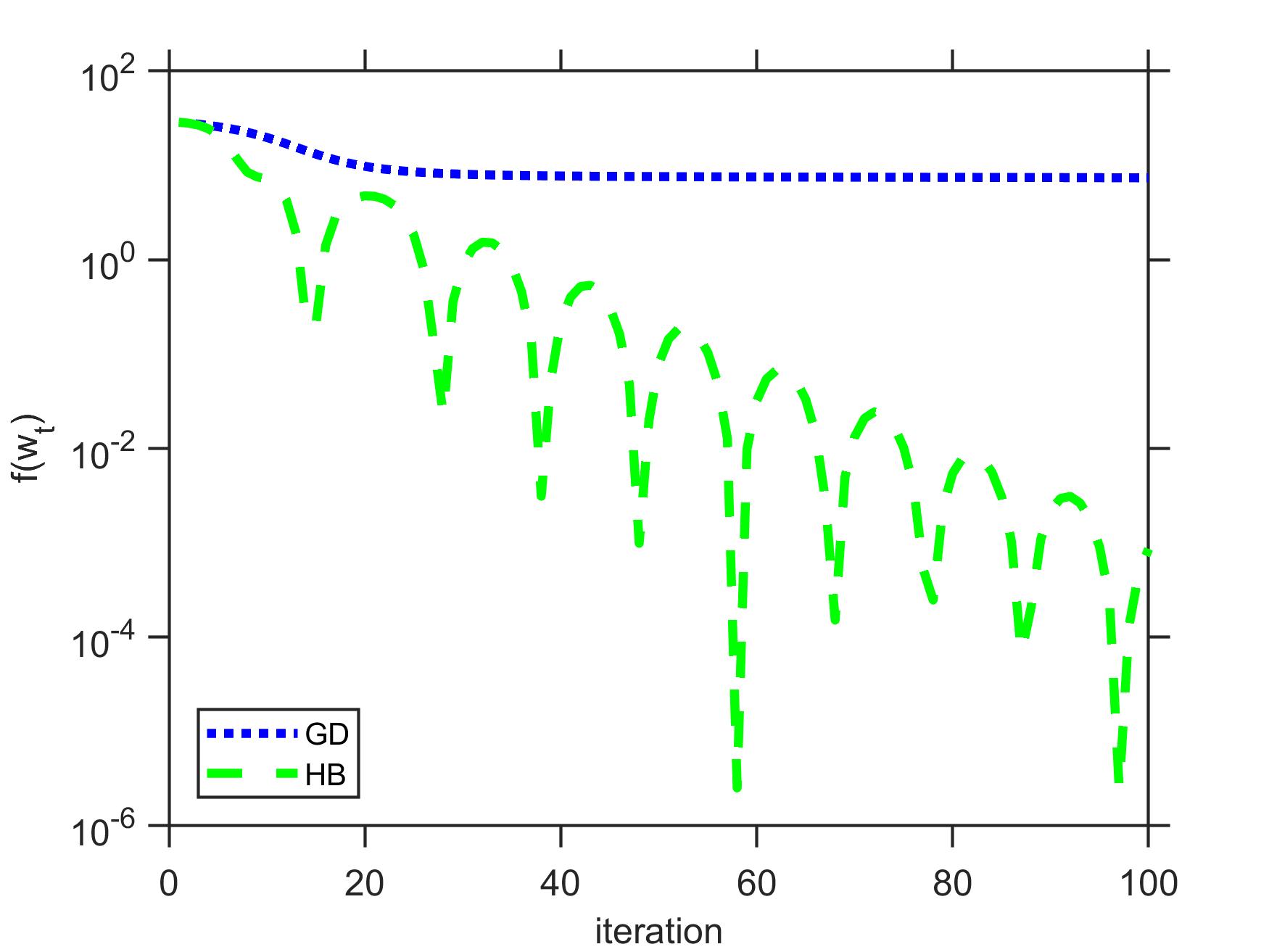}
     \caption{\footnotesize  (Left) function value $f(w)=w^{2}+ 4 sin^{2}(w)$ v.s. $w$.
     (Right) function value $f(w_t)$ v.s. iteration $t$ of
     solving $\min_{w} f(w)$ by GD and HB. Both algorithms were initialized at $w=-5$.} \label{fig:syn}
\end{figure}

\subsection{Example 2}
Consider minimizing
\begin{equation} \label{obj:syn}
f(w) := w^2 + \alpha \sin^2(w), 
\end{equation}
where $\alpha > 0$ is a problem parameter.
The function $f(w)$ was considered in \citet{KNS17}, where the authors claim that it satisfies PL.
Figure~\ref{fig:syn} plots the function and compares HB to GD for minimizing $f(w)$ in \eqref{obj:syn}. One can see from the sub-figure on the left that the function is non-convex, but the average Hessian towards the optimal point appears to be positive definite --- the non-convexity could be averaged out, which suggests that HB can converge faster than GD according to our main theorem. The sub-figure on the right confirms this empirically, while the following lemma provably shows that the \ao condition indeed holds.

\begin{lemma} \label{lem:syn}
Let $\alpha \in (1,4.34]$.
The function $f(w)$ in \eqref{obj:syn} is twice differentiable and has $L=2+2\alpha$-Lipschitz, $L_H= 4 \alpha$-Lipschitz Hessian, and its global optimal point is $w_{*}=0$.
Furthermore, it is non-convex but satisfies $\mu$-PL 
with $\mu =\frac{ \left(2+ \frac{sin(2w)}{w} \alpha\right)^2}{2+2\alpha} \geq \frac{ \left( 2- 0.46 \alpha \right)^2}{2+2 \alpha}$ at any $w$ and the average Hessian is $H_f(w) = 2 + \frac{sin(2w)}{w} \alpha > 0$, where $|\frac{sin(2w)}{w}| \leq 0.46$. 
Therefore, the condition \avg{\lambda_*}{w_*} holds with $\lambda_{*} = 2-0.46 \alpha$. 
\end{lemma}

The proof of Lemma~\ref{lem:syn} is available in Appendix~\ref{app:lem:syn}.
Combining Lemma~\ref{lem:syn} and Corollary~\ref{cor}, 
one can show that
HB for optimizing \eqref{obj:syn} has an instantaneous rate $1 - \Theta\left(\frac{1}{\sqrt{\kappa}} \right) $ at all $t \geq t_{0}$ for some number $t_{0}$, where $\kappa = \frac{(2+2\alpha)^2}{(2-0.46\alpha)^2}$.
It is noted that Lemma~\ref{lem:syn} also shows that $\lambda_{*}  = \Theta(\mu)$.
So the accelerated linear rate $1 - \Theta\left(\frac{1}{\sqrt{\kappa}} \right)$
in the proposition is considered to be better than $1 - \Theta\left(\frac{\mu}{L} \right)$.

We can generalize the above case to a broader class of problems as follows.

\begin{lemma} \label{lem:broad}
Consider solving
\[
\textstyle
\min_{w \in \reals^d} F(w) := \min_{w \in \reals^d} f(w) + g(w).
\]
Suppose the function $F(w)$ satisfies assumption $\clubsuit$.
Denote $w_{*}$ a global minimizer of $F(w)$.
Assume that $f(w)$ is $2 \lambda_{*}$-strongly convex for some $\lambda_{*}>0$. If the average Hessian of the possibly non-convex $g(w)$ between $w$ and $w_{*}$ satisfies $\lambda_{{\min}}(H_g(w)) \geq - \lambda_{*}$. Then, $F(\cdot)$ satisfies \avg{\lambda_*}{w_*}.
\end{lemma}
The lemmas says that if a possibly non-convex function $F(\cdot)$ can be decomposed into a strongly convex part and a possibly non-convex part, and if the average Hessian towards $w_{*}$ of the non-convex part is not too negative-definite, then \ao holds.

\subsection{Example 3}

Following \citet{GSD20,GBL19}, we consider minimizing the following non-convex function:
\begin{equation} \label{deep}
f(u) = \frac{1}{N^2} \| u^{\odot N} - w_* \|^2,
\end{equation}
where $w_{*} \in \reals_{+}^{d}$,
$N$ is a positive integer,
and the notation $u^{\odot N} \in \reals^{d}$ represents a vector whose $i_{\mathrm{th}}$ element is $(u_{{i}})^N$.
The $N$-layer diagonal network $u^{\odot N}$
has been a subject for studying incremental learning of GD \citep{GSD20} and the interaction between the scale of the initialization and generalization \citep{WGSMGLSS20}. On the applied side, using the diagonal network helps achieve sparse recovery without the need of tuning regularization parameters under certain conditions \citep{VKR19}.
It is noted that for $N=2$, the set of global optimal points of \eqref{deep} is
$\mathcal{U}:= \{ \hat{u} \in \reals^d:  \hat{u}_i = \pm \sqrt{w_{*,i}}  \}.$
The following lemma shows that the PL condition and \ao hold \emph{locally} (but not globally).

\begin{lemma} \label{lem:deep}
($N=2$).
The function $f(u)$ in (\ref{deep}) is twice differentiable.
It is $L=3R^{2}$-smooth and has $L_{H}=6R$-Lipschitz Hessian
in the ball $\{ u: \|u\|_{\infty} \leq R\}$ for any finite $R>0$.
Furthermore, $f(u)$ is $\mu$-PL with $\mu = 2 \min_{i\in [d]} (u_i)^2$ 
at $u$, and the average Hessian $H_{f}(u)$ between $u$ and $\hat{u} \in \mathcal{U}$
satisfies $\lambda_{{\min}}( H_f(u) ) = \min_{i \in [d]} \left( (u_i)^{2} + u_i \hat{u}_i \right)$.
\end{lemma}

\begin{figure}[t]
\subfigure[Average Hessian]{\label{fig:a}\includegraphics[width=.48\linewidth]{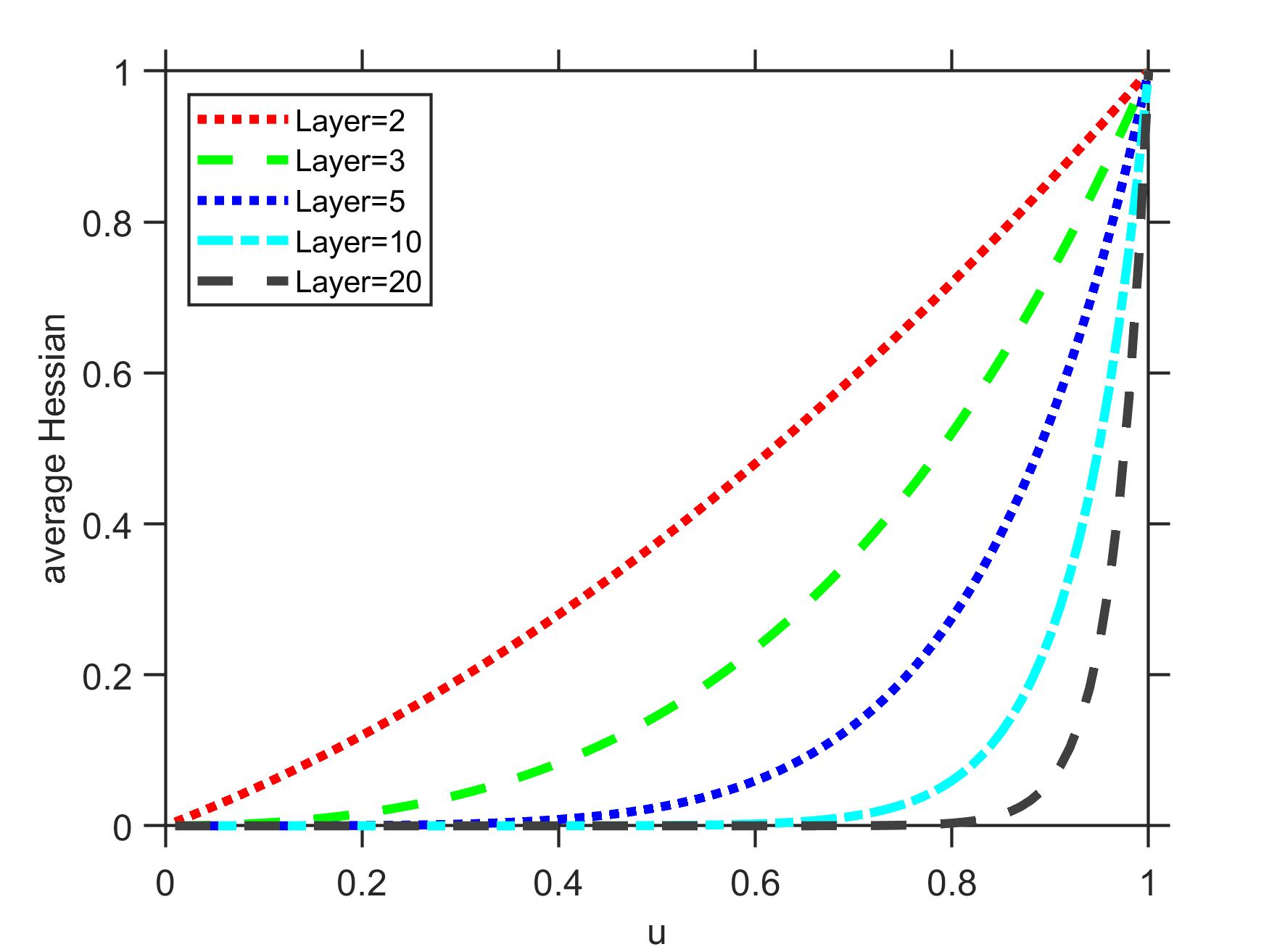}}
\subfigure[Obj. value vs. $t$]{\label{fig:b}\includegraphics[width=.48\linewidth]{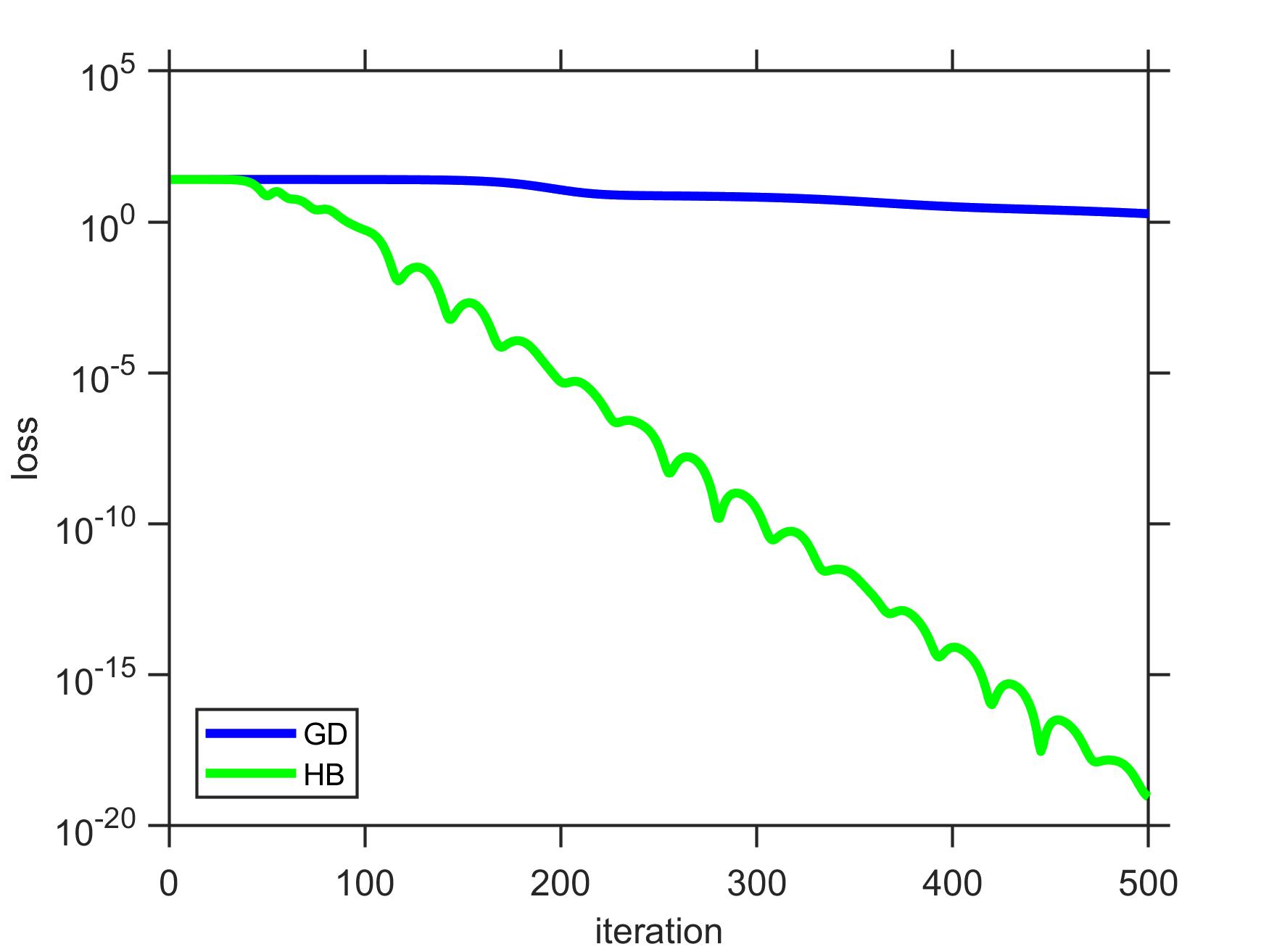}}
     \caption{ \footnotesize
     Subfigure (a): Average Hessian $H_{f}(u)$ between $u$ and $w_*=1$, where layer $N= \{ 2,3,5,10,20 \}$. A pattern is that adding more layer causes the need of iterate $u$ to be more closer to $w_*$ to have $\lambda_{\min}(H_{f}(u))$ sufficiently bounded away from $0$.
     Subfigure (b): Solving \eqref{deep} with $N=2$ by GD and HB.
     Both algorithms were initialized at $u = 10^{-2} \mathbf{1} \in \reals_{+}^{4} $ and the target is $w_*=[12, 6, 4, 3]^\top$, following the setup of \citet{GSD20}. }
  \label{fig:diag}
\end{figure}

Let $\sqrt{w_{*}}$ be the element-wise square root of $w_{*}$, which is one of the global minimizers. Then, Lemma~\ref{lem:deep} suggests that for all $u \in \reals^{d}_{+}$, the non-convexity w.r.t. $\sqrt{w_*}$ is averaged-out (and also PL holds).
Similar argument can be applied to any other global optimal points $\hat{u} \in \mathcal{U}$ in the sense that there exists a region where all the points inside the region satisfy PL and the non-convexity w.r.t. the corresponding global optimal point is averaged-out.
Lemma~\ref{lem:deep} implies that PL and \ao hold at all the points except those in 
$\mathcal{V} := \{ u \in \reals^d:  \exists i \in[d], s.t. \text{ } u_i = 0  \}.$
\begin{lemma} \label{lem:stay}
Let the initial point $u_0 = \alpha \mathbf{1} \in \reals^{d}_{+}$, where $\alpha$ satisfies $0< \alpha < \min_{i \in [d]} \sqrt{ w_{{*,i}} }$. Denote $L:= ~6 \max_{i \in [d]} w_{*,i}$.
Let $\eta = \frac{c_{\eta}}{L }$, where $c_{{\eta}} \in (0,1]$, and $\beta_{t} \leq \sqrt{ (1-\frac{2c_{\eta}^2 \tilde{c} \alpha^2 }{L} )\left( 1- 4 c_{\mu} \frac{ 2 \alpha^2 }{L} \right) } $ for any $\tilde{c} \in (0,1]$ and $c_{{\mu}} \in (0,\frac{1}{4}]$.
Then, for minimizing \eqref{deep}, the iterate $u_{t}$ generated by HB satisfies $0 < \alpha < u_{{t,i}} < \sqrt{2 w_{*,i} - \alpha^2 }$ for all $t$ and all $i \in [d]$.
\end{lemma}
Lemma~\ref{lem:stay} shows that HB always stays away from the points in $\mathcal{V}$ under certain conditions of the step size and the momentum parameter when initialized properly, and hence it remains in a region where PL and $\textsc{average out}$ hold. Therefore, our main theorem and the corollary can apply.
The proof of Lemma~\ref{lem:deep} and Lemma~\ref{lem:stay} is in Appendix~\ref{app:ex3}, where we have more discussions.
For the case of $N>2$, we also show that the PL condition holds locally in Appendix~\ref{app:lem:deep}. However, the average Hessian $H_f(u)$ does not have a simple analytical form. With the computer-aid analysis, we plot the average Hessian towards $w_{*}$ for some number of layers $N$ when the target is $w_{*}=1$ in
Sub-figure (a) of 
 Figure~\ref{fig:diag}. 
Sub-figure (b) 
of the same figure
shows HB converges significantly faster than GD empirically.

\section{Acknowledgment}
The authors thank the anonymous reviewers for their constructive feedback.
The authors also thank Molei Tao and Sinho Chewi for discussions.
B. Hu was supported by the National Science Foundation (NSF) award CAREER-2048168.

\bibliography{icml_HBPL}
\bibliographystyle{icml2022}

\clearpage
\appendix

\onecolumn

\section{More discussions about prior works of HB} \label{app:HB}

\begin{algorithm}[t]
\begin{algorithmic}[1]
\small
\caption{Heavy Ball (Equivalent Version 2)
} \label{alg:HB1}{}
\STATE Required: the step size $\eta$ and the momentum parameter $\beta_{t}$.
\STATE Init: $w_{0} \in \reals^d $ and $m_{-1} = 0 \in \reals^d$.
\FOR{$t=0$ to $T$}
\STATE Given current iterate $w_t$, compute gradient $\nabla f(w_t)$.
\STATE Update momentum $m_t = \beta_t  m_{t-1} +  \nabla f(w_t)$.
\STATE Update iterate $w_{t+1} = w_t - \eta m_t$.
\ENDFOR
\end{algorithmic}
\end{algorithm}

\subsection{Spectral norm of $A_{t}$ in \eqref{dynamic} satisfies $\| A_t \|_2 \geq 1$}

Let us recall the definition $A_t:=\begin{bmatrix}
I_{d} - \eta H_t + \beta I_{d} & - \beta  I_{d}   \\
I_{d} & 0_{d} 
\end{bmatrix}$ in HB's dynamic \eqref{dynamic}.
We have Lemma~\ref{lem:At1} below, which shows that the spectral norm of $A_{t}$ is not less than $1$. This implies that the spectral norm $\| A_t \|_{2}$ is not a useful quantity to look at when we try to derive the convergence rate of HB.

\begin{lemma} \label{lem:At1}
The spectral norm satisfies $\| A_t \|_2 \geq 1$
for any step size $\eta \leq \frac{1}{L}$ and any momentum parameter $\beta \in [0,1]$.
\end{lemma}

\begin{proof}
For brevity, in this proof we suppress the subscript $t$, i.e., we write
$A:=\begin{bmatrix}
I_{d} - \eta H + \beta I_{d} & - \beta  I_{d}   \\
I_{d} & 0_{d} 
\end{bmatrix}$.

Let $U{\rm Diag}([\lambda_1,\dots,\lambda_d]) U^\top$ be the eigen-decomposition of $\eta H$,
where $\lambda_1,\dots,\lambda_d$ are the eigenvalues of $\eta H$.
Observe that each $\lambda_{i}$ satisfies $\lambda_i \leq 1$ since $\eta \leq \frac{1}{L}$.

We have
\begin{align}
    A = 
\begin{bmatrix}
U& 0\\0& U
\end{bmatrix}
\begin{bmatrix} 
(1 + \beta) I_d - {\rm Diag}([\lambda_1,\dots,\lambda_d])  & - \beta I_d \\
I_d & 0 
\end{bmatrix}
\begin{bmatrix}
U^\top& 0\\0& U^\top
\end{bmatrix}.
\end{align}
Let $\tilde{U} := \begin{bmatrix}
U& 0\\0& U
\end{bmatrix}$. Then, after applying some permutation matrix $\tilde{P}$, the matrix $A$ can be further written as
\begin{align}\label{decompose1}
A = \tilde{U}\tilde{P} \Sigma \tilde{P}^T\tilde{U}^\top,
\end{align}
where $\Sigma$ is a block diagonal matrix consisting of $d$ 2-by-2
matrices $\tilde{\Sigma}_i := 
\begin{bmatrix}
1+\beta- \lambda_i& -\beta\\1& 0
\end{bmatrix}$. 

Since $\tilde{U}$ and $\tilde{P}$ are unitary matrices,
it does not affect the spectral norm $\|A \|_2$. 
That is, $\| A \|_2  = \| \Sigma \|_{2}$.
Moreover, as $\Sigma \in \reals^{{2d \times 2d }}$ is a block diagonal matrix,
we have $\| \Sigma \|_2 = \max_{i \in [d] } \| \Sigma_i \|_2$. 
Therefore, it suffices to show that the spectral norm of each sub-matrix $\Sigma_{i} \in \reals^{2 \times 2}$ satisfies $\| \Sigma_i \|_{2} \geq 1$.

By definition of the spectral norm, we have
\begin{equation}
\begin{split}
\| \Sigma_i \|_2 & 
= \sqrt{ \lambda_{\max}\left( \begin{bmatrix} 1 + \beta - \lambda_i  & 1 \\ - \beta & 0   \end{bmatrix} \begin{bmatrix} 1 + \beta - \lambda_i  & -\beta \\ 1 & 0   \end{bmatrix} \right) }
\\ & = \sqrt{ \lambda_{\max}\left( \underbrace{ \begin{bmatrix} (1+\beta-\lambda_i)^2 +1 & -\beta(1+\beta - \lambda_i) \\ -\beta(1+\beta -\lambda_i) & \beta^2 \end{bmatrix} }_{:=R_i}  \right)  }.
\end{split}
\end{equation}
The characteristic polynomial of $R_{i}$ is $x^{2} - \left( \beta^2 + (1+\beta-\lambda_i)^2 + 1  \right) x + \beta^{2} (  (1+\beta-\lambda_i)^2 + 1  ) - \beta^2 (1+\beta -\lambda_i)^2  $, which has two roots $x =  \frac{ \beta^2 + (1+\beta-\lambda_i)^2 + 1 
\pm \sqrt{  \left( \beta^2 + (1+\beta-\lambda_i)^2 + 1  \right)^2  - 4 \beta^2 }
  }{2}$. To show that the largest root $x \geq 1$ for any $\lambda_{i} \leq 1$, it suffices to show that $\left( \beta^2 + (1+\beta-\lambda_i)^2 + 1  \right)^2  - 4 \beta^2 \geq 1$ for any $\lambda_{i} \leq 1$. By simple calculations, it can be shown that the latter is true.  So $\lambda_{{\max}}(R_i)\geq 1$ and hence $\| \Sigma_i \|_2 \geq 1$. Since this holds for all $i$, we can conclude that $\| A \|_{2} \geq 1$. 
\end{proof}

\subsection{More discussions about \citet{WLA21}}

Here we discuss the limitation of the approach by \citet{WLA21} further.
Denote $\xi_{t}:= w_{t} - w_{*}$.
The way that \citet{WLA21} proposed is first by rewriting \eqref{dynamic} as
\begin{equation} \label{eq:m1}
\begin{split}
\begin{bmatrix}
\xi_{t+1} \\
\xi_{t} 
\end{bmatrix}
 &
=
\begin{bmatrix}
I_{d} - \eta H_0 + \beta I_{d} & - \beta  I_{d}   \\
I_{d} & 0_{d} 
\end{bmatrix}
\begin{bmatrix}
\xi_{t} \\
\xi_{t-1} 
\end{bmatrix}
+
\begin{bmatrix}
\varphi_{t} \\ 0_{d}
\end{bmatrix}
,
\end{split}
\end{equation}
where $\varphi_{t} = \eta \left( H_t - H_{0} \right) (w_{t} - w_*)$.
Recursively expanding the above equation leads to 
\begin{equation} \label{eq:m2}
\begin{bmatrix}
\xi_{t+1} \\
\xi_{t} 
\end{bmatrix}
=  
A^{t+1}
\begin{bmatrix}
\xi_{0} \\
\xi_{-1} 
\end{bmatrix}
+ \sum_{s=0}^{t}  A^{t-s} \begin{bmatrix}
\varphi_s \\
0 
\end{bmatrix}
.
\end{equation}
The first term on the r.h.s. of \eqref{eq:m2} is the size of the matrix-power-vector product and hence can be bounded as
$\textstyle
\left \|
A^{t+1}
\begin{bmatrix}
w_0 - w_* \\
w_{-1}- w_* 
\end{bmatrix}
\right\|_2
\leq 4 \sqrt{\kappa} \left( 1 - \frac{1}{2 \sqrt{\kappa}} \right)^{t+1}
\|
\begin{bmatrix}
w_0 - w_* \\
w_{-1}- w_* 
\end{bmatrix} \|,
$
which decays at an accelerated linear rate. For the second term on the r.h.s. of \eqref{eq:m2}, \citet{WLA21} show that if $\varphi_{t}$ on \eqref{eq:m1} is sufficiently small for all $t$, then the second term would not be dominant so that an accelerated linear rate of the convergence still holds. However, the constraint only allows $H_{t}$ to be a slight deviation from $H_{0}$, which could be seen from their theorem statements, e.g., Theorem 6 and 8 in the paper.

\subsection{Some clarifications}

\citep{LRP16} give a non-convergent example of HB, which is a piecewise linear function satisfies
\begin{equation}
\nabla f(w) = \begin{cases}  25 w, \quad  \text{ if } w < 1, \\ w + 24 , \quad \text{ if } 1 \leq w < 2  \\ 25 w -24 , \quad \text{ if } w \geq 2 \end{cases}.
\end{equation}
It should be clarified that the divergence is under \emph{a specific choice} of the momentum parameter and the step size, which in no way implies that HB would be doomed to diverge for minimizing this function under other configurations of the parameters. 
Since the function is strongly convex smooth (but not twice differentiable), by invoking the results of prior works \cite{KDP18,GFJ15,WAL21}, one can show that HB converges under appropriate choices of the momentum parameter and the step size for minimizing this function.

Finally, given a dynamic $\xi_{t+1} = A \xi_{t}$, 
where $A$ is a fixed real matrix whose spectral radius satisfies $\rho(A) < 1$,
one can show that there exists a norm $\| \cdot \|_{P} = \langle \cdot, P \cdot \rangle$ for some positive definite matrix $P$ such that
$\| \xi_{t+1} \|_P \leq ( \rho(A) + \bar{\epsilon} ) \| \xi_t \|_{P}$ holds for any $\bar{\epsilon} > 0$. The caveat is to find the matrix $P$ (which depends on the chosen $\bar{\epsilon}$) to make the norm
$\| \cdot \|_{P}$ explicit. When converting the bound in $\| \cdot \|_{P}$ to that in terms of the $l_{2}$-norm, one would get $\| \xi_{t+1} \|_2 \leq C ( \rho(A) + \bar{\epsilon} )^t \| \xi_0 \|_{2}$, where $C>0$ is a constant that depends on the the condition number of $P$, which in turn depends on the chosen $\bar{\epsilon}$.
We refer the reader to a lecture note \cite{HU20} for further details.

\section{Proof of Lemma~\ref{lem:diagonal}} \label{app:lem1}

\noindent
\textbf{Lemma~\ref{lem:diagonal} (Full-version):}[\citet{WLA21}]
\textit{
Consider a matrix
\begin{equation}
A:=
\begin{bmatrix} 
(1 + \beta) I_d -  H  & - \beta I_d\\ 
I_d & 0
\end{bmatrix}
\in \reals^{2 d \times 2 d},
\end{equation}
where $H \in \reals^{{d \times d}}$ is a symmetric matrix and has eigenvalues $\lambda_{1} \geq \lambda_2 \geq \dots \geq \lambda_{i} \geq \dots \geq \lambda_{d}$.
Suppose $\beta$ satisfies $1\geq \beta > \left( 1 - \sqrt{\lambda_i}  \right)^2$ for all $i \in [d]$. Then, we have
$A$ is diagonalizable with respect to the complex field $\mathbb{C}$ in $\mathbb{C}^{2d \times 2d}$, i.e.,
\begin{equation}
(\textbf{spectral decomposition}) \quad A = PDP^{-1}
\end{equation}
for some matrix
$P = \tilde{U}\tilde{P} Q$,
where
$\tilde{U}$ and $\tilde{P}$ are some orthogonal matrices,
$Q ={\rm Diag}(Q_1,\dots,Q_d) \in \mathbb{C}^{{2d \times 2d}}$ is a block diagonal matrix,
and $D \in \mathbb{C}^{{2d \times 2d}}$ is a diagonal matrix.
Specifically, the $i_{{th}}$ block diagonal of $Q$ is $Q_i = [q_i,\bar{q_i}] \in \mathbb{C}^{{2 \times 2}}$, where 
$q_i = \begin{bmatrix}z_i \\ 1 \end{bmatrix} $ and $\bar{q_i}= \begin{bmatrix}\bar{z}_i \\ 1 \end{bmatrix}$ are the eigenvectors of $\tilde{\Sigma}_i := \begin{bmatrix}
1+\beta-\lambda_i& -\beta\\1& 0
\end{bmatrix} \in \reals^{{2 \times 2}}$ with corresponding eigenvalues $z_{i}$ and $\bar{z}_{i}$ respectively.
Furthermore, the diagonal matrix $D$ is
\begin{align}
D ={\rm Diag}\left(\begin{bmatrix}
    z_{1}&0\\0&\bar{z}_{1}
    \end{bmatrix},\dots,\begin{bmatrix}
    z_{d}&0\\0&\bar{z}_{d}
    \end{bmatrix}\right) \in \mathbb{C}^{{2d \times 2d}} 
\end{align}
and $|z_{i}| = \sqrt{\beta}$, which means that the magnitude of
each diagonal element on the sub-matrix 
$\begin{bmatrix}
    z_{i}&0\\0&\bar{z}_{i}
\end{bmatrix} \in \mathbb{C}^{{2 \times 2}} $ 
is $\sqrt{\beta}$ and hence $\| D \|_{2} = \sqrt{ \beta}$.
}
\begin{proof}
We replicate the proof of \citet{WLA21} here for the completeness.

Let $U{\rm Diag}([\lambda_1,\dots,\lambda_d]) U^\top$ be the eigen-decomposition of $H$. Then 
\begin{align}
    A = 
\begin{bmatrix}
U& 0\\0& U
\end{bmatrix}
\begin{bmatrix} 
(1 + \beta) I_d - {\rm Diag}([\lambda_1,\dots,\lambda_d])  & - \beta I_d \\
I_d & 0 
\end{bmatrix}
\begin{bmatrix}
U^\top& 0\\0& U^\top
\end{bmatrix}.
\end{align}
Let $\tilde{U} := \begin{bmatrix}
U& 0\\0& U
\end{bmatrix}$. Then, after applying some permutation matrix $\tilde{P}$, the matrix $A$ can be further simplified into 
\begin{align}\label{decompose1}
A = \tilde{U}\tilde{P} \Sigma \tilde{P}^T\tilde{U}^\top,
\end{align}
where $\Sigma$ is a block diagonal matrix consisting of $d$ 2-by-2
matrices $\tilde{\Sigma}_i := 
\begin{bmatrix}
1+\beta- \lambda_i& -\beta\\1& 0
\end{bmatrix}$. The characteristic polynomial of $\tilde{\Sigma}_i$ is $x^2 - (1+\beta -\lambda_i)x +\beta$. Hence it can be shown that when $\beta > (1-\sqrt{ \lambda_i})^2$ then the roots of polynomial are conjugate and have magnitude $\sqrt{\beta}$. These roots are exactly the eigenvalues of $\tilde{\Sigma}_i \in \reals^{2 \times 2}$. On the other hand, the corresponding eigenvectors $q_i,\bar{q}_i$ are also conjugate to each other as $\tilde{\Sigma}_i \in \reals^{2 \times 2}$ is a real matrix. As a result, $\Sigma \in \reals^{2d \times 2d}$ admits a block eigen-decomposition as follows,
\begin{align}\label{decompose2}
    \Sigma = & {\rm Diag}(\tilde{\Sigma}_i,\dots,\tilde{\Sigma}_d)\nonumber\\
    =&
    {\rm Diag}(Q_1,\dots,Q_d)
    {\rm Diag}\left(\begin{bmatrix}
    z_{1}&0\\0&\bar{z}_{1}
    \end{bmatrix},\dots,\begin{bmatrix}
    z_{d}&0\\0&\bar{z}_{d}
    \end{bmatrix}\right){\rm Diag}(Q^{-1}_1,\dots,Q^{-1}_d),
\end{align}
where $Q_i= [q_i,\bar{q}_i] \in \mathbb{C}^{2 \times 2}$ and $z_{i}, \bar{z}_{i}$ are eigenvalues of $\tilde{\Sigma}_i:= \begin{bmatrix}
1+\beta-\lambda_i& -\beta\\1& 0
\end{bmatrix}$,
since they are conjugate by the condition on $\beta_{i}$. 
The eigenvalues satisfy 
\begin{align}\label{val}
    z_i + \bar{z}_i = 2\Re{z_i} &= 1+\beta -\lambda_i, \\
    z_i\bar{z}_i &= |z_i|^2 = \beta. \label{val22}
\end{align}
On the other hand, the eigenvalue equation $\tilde{\Sigma}_iq_i = z_i q_i$ together with (\ref{val}) implies $q_i = \begin{bmatrix}z_i \\ 1 \end{bmatrix}$.
Denote $Q:={\rm Diag}(Q_1,\dots,Q_d)$ and
\begin{align}
D:={\rm Diag}\left(\begin{bmatrix}
    z_{1}&0\\0&\bar{z}_{1}
    \end{bmatrix},\dots,\begin{bmatrix}
    z_{d}&0\\0&\bar{z}_{d}
    \end{bmatrix}\right).
\end{align}
By combining (\ref{decompose1}) and (\ref{decompose2}), we have
\begin{align}
    A & = P {\rm Diag}\left(\begin{bmatrix}
    z_{1}&0\\0&\bar{z}_{1}
    \end{bmatrix},\dots,\begin{bmatrix}
    z_{d}&0\\0&\bar{z}_{d}
    \end{bmatrix}\right) P^{-1}
     = P D P^{-1} ,
\end{align}
where 
\begin{align} \label{def:P}
P = \tilde{U}\tilde{P}Q,
\end{align} by the fact that $\tilde{P}^{-1} = \tilde{P}^\top$ and $\tilde{U}^{-1}=\tilde{U}^{\top}$.

\end{proof}

\section{Proof of Lemma~\ref{vb}} \label{app:lem2}

\noindent
\textbf{Lemma~\ref{vb}}
\textit{
Assume we have $1\geq \beta_t > (1-\sqrt{\eta \lambda_{t,i}})^2$ 
for all $i \in [d]$ and that the co-diagonalization condition $\spadesuit$ holds.
Then, the \textbf{instantaneous rate} $\| \Psi_t\|_2$ at $t$ satisfies:\\
(I) If $\lambda_{t,i} \geq \lambda_{t-1,i}$, then
\begin{equation}  \label{vbb1}
\begin{split}
& \|\Psi_t\|_2
 = 
  \sqrt{\beta_t} \times
\underbrace{ \left( \max_{i \in [d]} \sqrt{1+ \frac{\eta \lambda_{t,i} -\eta \lambda_{t-1,i}}{ (1+\sqrt{\beta_t})^2 - \eta \lambda_{t,i}  } }
 + \mathbbm{1}\{\beta_t \neq \beta_{t-1}\} \phi_{t,i}
\right) }_{\text{extra factor 1} }.
\end{split}
\end{equation}
(II) If $\lambda_{t-1,i} \geq \lambda_{t,i}$, then
\begin{equation} \label{vbb2}
\begin{split}
\|\Psi_t\|_2
& = 
\sqrt{\beta_t} \times
\underbrace{ \left( \max_{i \in [d]}
\sqrt{1+ \frac{\eta \lambda_{t-1,i} - \eta \lambda_{t,i}}{ \eta \lambda_{t,i}  - (1-\sqrt{\beta_t })^2  } } 
 + \mathbbm{1}\{\beta_t \neq \beta_{t-1}\} \phi_{t,i}
\right) }_{\text{extra factor 2} }.
\end{split}
\end{equation}
The term $\phi_{{t,i}}$ in \eqref{vbb1} and \eqref{vbb2} can be bounded as
\begin{equation} \label{phidef}
\begin{split}
\phi_{t,i} & \leq
3 \sqrt{ \frac{ | \beta_{t-1} - \beta_t| }{| 4\beta_t - (1+\beta_t - \eta \lambda_{t,i})^2| } } +  3  \sqrt{  \frac{ \sqrt{ | \beta_{t-1} - \beta_t| | \eta \lambda_{t,i} - \eta \lambda_{t-1,i}| } }{ |4\beta_t - (1+\beta_t -\eta \lambda_{t,i})^2| }   }.
\end{split}
\end{equation}
}

\begin{proof} 
By Lemma~\ref{lem:diagonal}, we have 
\[
\|\Psi_t \|_2 := \|D_tP_t^{-1}P_{t-1}\|_2 \leq \sqrt{\beta_t}\|P_t^{-1}P_{t-1}\|_{2},
\]
Based on the decomposition from Lemma~\ref{lem:diagonal}, 
we have $P_t = \tilde{U}_t \tilde{P}_t Q_t$,
where $Q_t:={\rm Diag}(Q_{t,1},\dots,Q_{t,d}) $,
and $\tilde{U}_t$ and $\tilde{P}_t$ are some orthogonal matrices.

Then, $\| P_t^{-1}P_{t-1} \|_2 = \| Q_t^{-1} M_t Q_{t-1} \|_2$, where $M_t = (\tilde{P}_t)^{-1} (\tilde{U}_t)^{-1} \tilde{U}_{t-1}\tilde{P}_{t-1}$ is an orthogonal matrix.
When the condition $\spadesuit$ holds, it suffices to analyze $\| Q_t^{-1} Q_{t-1} \|_2$ for bounding $\| P_t^{-1}P_{t-1} \|_2$, which is conducted by exploiting the fact that $Q_t$ is block-diagonal as follows.

In the following, we consider a fixed $i \in [d]$. For brevity, in the following we denote 
\begin{mdframed}
\begin{align} 
Q_{t}  & := Q_{t,i}  \notag \\
\lambda_{t} & := \eta_t \lambda_{{t,i}} \label{brev1} \\
z_t &:= z_{t,i} \notag \\
\bar{z}_t &:= \bar{z}_{t,i}. \notag
\end{align}
\end{mdframed}
By Lemma~\ref{lem:diagonal} (full version) and (\ref{val}), we have
\begin{align}
    Q_t &= \begin{bmatrix}
    z_t & \bar{z}_t\\
    1 & 1
    \end{bmatrix},\\
    \Re{z_t} &= \frac{1+\beta_t-\lambda_t}{2}, \label{rzt} \\
    \Im{z_t}&=\sqrt{\beta_t - \left(\frac{1+\beta_t-\lambda_t}{2}\right)^2}.\label{izt}
    \end{align}
    These quantities would play an essential role for the remaining calculations.
    On the other hand,
    \begin{align}
    G_t:= Q_{t}^{-1}Q_{t-1} &= \frac{1}{z_t-\bar{z}_t}
    \begin{bmatrix}
    1 & -\bar{z}_t\\
    -1 & z_t
    \end{bmatrix}
    \begin{bmatrix}
    z_{t-1} & \bar{z}_{t-1}\\
    1 & 1
    \end{bmatrix}
    \\
    & = \frac{1}{z_t-\bar{z}_t}\begin{bmatrix}
    z_{t-1}-\bar{z}_t & \bar{z}_{t-1}-\bar{z}_t\\
    -z_{t-1}+z_t & -\bar{z}_{t-1} + z_t
    \end{bmatrix}.
\end{align}
Recall the definition of the spectral norm, we have $\|G_t\|_2 = \sqrt{\lambda_{\max}({G_tG_t^*})}$. 
Observe that $G_t$ is of the form 
$\begin{bmatrix}
a & \bar{b}\\-b & - \bar{a} 
\end{bmatrix}$, where
\begin{align} 
    a &= \frac{1}{2j\Im{z_t}}(z_{t-1}-\bar{z_t}), \label{ab1}\\
    b &= \frac{1}{2j\Im{z_t}}(z_{t-1}-z_t). \label{ab2}
\end{align}
This results in that $G_tG_t^*$ be of the form $\begin{bmatrix}
{a}' & \bar{{b}'}\\ {b}' &  {a}' 
\end{bmatrix}$, where $a' = |a|^2 + |b|^2$ and $b'=-2\bar{a}b$. Let the eigenvalues of $G_tG_t^*$ be $\xi_1$ and $\xi_2$, then $\xi_1 + \xi_2 = 2a'$, $\xi_1 \xi_2 = a'^2-|b'|^2$. Therefore, it is clear that $\xi_1= a' + |b'|$, $\xi_2 = a' - |b'|$, and 
\begin{equation} \label{GM}
\lambda_{\max}(G_tG_t^*) = a' + |b'|.
\end{equation}
Our remaining task is to calculate $a' = |a|^2 + |b|^2$ and $|b'|=2|a||b|$.

First we have 
\begin{equation} \label{eqa}
\begin{split} 
& a'=|a|^2 + |b|^2 \overset{\eqref{ab1},\eqref{ab2}}{=}
    \frac{1}{4(\Im {z_t})^2}(2|z_{t-1}|^2 + 2|z_t|^2 - 4\Re{z_t}\Re{z_{t-1}})
    \\
    &\overset{\eqref{val},\eqref{val22}}{=}\frac{2\beta_t + 2 \beta_{t-1} -  (1+\beta_t-\lambda_t)(1+\beta_{t-1}-\lambda_{t-1})}{4\beta_t - (1+\beta_t - \lambda_t)^2}
  \\  & \quad =\frac{4\beta_t -  (1+\beta_t-\lambda_t)(1+\beta_{t}-\lambda_{t-1})}{4\beta_t - (1+\beta_t - \lambda_t)^2} 
  + \frac{ 2(\beta_{t-1} - \beta_t) - (\beta_{t-1} - \beta_t ) ( 1 + \beta_t - \lambda_t) }{4\beta_t - (1+\beta_t - \lambda_t)^2}. 
\end{split}
\end{equation}
Second we also have $|b'|=2|a||b|=$
\begin{align} \label{eqb}
    &\frac{\sqrt{(\Re{z_{t-1}}-\Re{z_t})^4 + 2 (\Re{z_{t-1}}-\Re{z_t})^2((\Im{z_{t-1}})^2+(\Im{z_t})^2) + ((\Im{z_t})^2-(\Im{z_{t-1}})^2)^2 }}{2(\Im{z_t})^2}.
\end{align}
For $\Re{z_{t-1}}-\Re{z_t}$ in \eqref{eqb}, we have
\begin{align}
\Re{z_{t-1}}-\Re{z_t} \overset{\eqref{rzt}}{ =} \left( \frac{\beta_{t-1}-\beta_t}{2}  + \frac{\lambda_t - \lambda_{t-1}}{2}  \right). \label{i1}
\end{align}
For $(\Im{z_{t-1}})^2+(\Im{z_t})^2$ in \eqref{eqb}, we have
\begin{equation} \label{i2}
\begin{split} 
 (\Im{z_{t-1}})^2+(\Im{z_t})^2
& \overset{\eqref{izt}}{ =} \beta_{t-1} - \left(  \frac{1+\beta_{t-1} - \lambda_{t-1}}{ 2}  \right)^2
+ \beta_{t} - \left(  \frac{1+\beta_{t} - \lambda_t}{ 2}  \right)^2 
\\ &
= 2\beta_t -\left(\frac{1+\beta_t-\lambda_t}{2}\right)^2-\left(\frac{1+\beta_t-\lambda_{t-1}}{2}\right)^2 
\\ & \quad + ( \beta_{t-1} - \beta_t )
- 2 \left(  \frac{1+\beta_t-\lambda_{t-1}}{2}  \right) \left( \frac{\beta_{t-1} - \beta_t}{2}  \right) - \left( \frac{\beta_{t-1} - \beta_t}{2}  \right)^2. 
\end{split}
\end{equation}
For $(\Im{z_{t}})^2-(\Im{z_{t-1}})^2$ in \eqref{eqb}, we have
\begin{equation}
\begin{split}
 (\Im{z_{t}})^2-(\Im{z_{t-1}})^2
& = 
\beta_{t} - \left(  \frac{1+\beta_{t} - \lambda_t}{ 2}  \right)^2
-
\beta_{t-1} + \left(  \frac{1+\beta_{t-1} - \lambda_{t-1}}{ 2}  \right)^2
\\ &
= \beta_{t} - \beta_{t-1}
+ \left( \frac{1+\beta_{t-1} - \lambda_{t-1}}{ 2} + \frac{1+\beta_{t} - \lambda_{t}}{ 2}   \right) \left( \frac{1+\beta_{t-1} - \lambda_{t-1}}{ 2} - \frac{1+\beta_{t} - \lambda_{t}}{ 2}  \right)
\\ &
= \beta_t - \beta_{t-1}
+ \left( \frac{1+\beta_{t} - \lambda_{t-1}}{ 2} + \frac{1+\beta_{t} - \lambda_{t}}{ 2}    \right) \left( \frac{\lambda_t - \lambda_{t-1} + \beta_{t-1} - \beta_t  }{2 }   \right)
\\ &
\qquad + \left( \frac{\beta_{t-1} - \beta_{t} }{ 2 }    \right)
\left(  \frac{ \beta_{t-1} - \beta_{t} +\lambda_t - \lambda_{t-1}  }{2}   \right)
\\ &
= \beta_t - \beta_{t-1}
+ \left( \frac{1+\beta_{t} - \lambda_{t-1}}{ 2} + \frac{1+\beta_{t} - \lambda_{t}}{ 2}    \right) \left( \frac{\lambda_t - \lambda_{t-1}}{2}  \right)
\\ & \qquad + \left( \frac{\beta_{t-1} - \beta_{t} }{ 2 }    \right)
\left( \frac{2+\beta_{t-1} + \beta_{t} - 2 \lambda_{t-1} }{ 2 }   \right) \label{i3}
\\ &
= \left( \beta_t - \beta_{t-1} \right)
\left( 1  - \frac{2 + \beta_{t-1} + \beta_{t} - 2 \lambda_{t-1} }{ 4 }   \right)
+ \left( \frac{1+\beta_{t} - \lambda_{t-1}}{ 2} + \frac{1+\beta_{t} - \lambda_{t}}{ 2}    \right) \left( \frac{\lambda_t - \lambda_{t-1}}{2}  \right).
\end{split}
\end{equation}
Combing \eqref{eqb}, \eqref{i1}, \eqref{i2}, \eqref{i3}, we obtain
\begin{equation} \label{b}
\begin{split}
|b'| &=  \sqrt{(\Re{z_{t-1}}-\Re{z_t})^4 + 2 (\Re{z_{t-1}}-\Re{z_t})^2((\Im{z_{t-1}})^2+(\Im{z_t})^2) + ((\Im{z_t})^2-(\Im{z_{t-1}})^2)^2 }
\\ & = \frac{ \sqrt{ \textcircled{1} + \textcircled{2} } }{2(\Im{z_t})^2},
\end{split}
\end{equation}
where we defined 
\begin{equation}
\textstyle
\textcircled{1}:=
\left(\frac{\lambda_t-\lambda_{t-1}}{2}\right)^4+2\left(\frac{\lambda_t-\lambda_{t-1}}{2}\right)^2\left(2\beta_t-\left(\frac{1+\beta_t-\lambda_t}{2}\right)^2-\left(\frac{1+\beta_t-\lambda_{t-1}}{2}\right)^2\right) + \left(\frac{1+\beta_t-\lambda_t}{2}+\frac{1+\beta_t-\lambda_{t-1}}{2}\right)^2\left(\frac{\lambda_t-\lambda_{t-1}}{2}\right)^2,
\end{equation}
and also defined 
\begin{equation}
\begin{split}
\textstyle \textcircled{2} := & \textstyle
 4 \left(  \frac{\beta_{t-1}-\beta_t}{2} \right)^3 \left(  \frac{\lambda_t-\lambda_{t-1}}{2} \right)
+ 6  \left(  \frac{\beta_{t-1}-\beta_t}{2} \right)^2 \left(  \frac{\lambda_t-\lambda_{t-1}}{2} \right)^2 
+ 4  \left(  \frac{\beta_{t-1}-\beta_t}{2} \right)^1 \left(  \frac{\lambda_t-\lambda_{t-1}}{2} \right)^3
  + \left(  \frac{\beta_{t-1}-\beta_{t}}{2} \right)^4
\\ &  \textstyle
 +
2 \left( \left( \frac{\beta_{t-1}-\beta_t}{2} \right)^2 
+ 2 \left( \frac{\beta_{t-1}-\beta_t}{2} \right) 
\left( \frac{\lambda_{t}-\lambda_{t-1}}{2} \right)
 \right)
 \left( 
  ( \beta_{t-1} - \beta_t )
- 2 \left(  \frac{1+\beta_t-\lambda_{t-1}}{2}  \right) \left( \frac{\beta_{t-1} - \beta_t}{2}  \right) - \left( \frac{\beta_{t-1} - \beta_t}{2}  \right)^2
 \right)
\\ &  \textstyle
 +
2 \left( \left( \frac{\beta_{t-1}-\beta_t}{2} \right)^2 
+ 2 \left( \frac{\beta_{t-1}-\beta_t}{2} \right) 
\left( \frac{\lambda_{t}-\lambda_{t-1}}{2} \right)
 \right)
 \left( 
2\beta_t -(\frac{1+\beta_t-\lambda_t}{2})^2-(\frac{1+\beta_t-\lambda_{t-1}}{2})^2 
 \right)
\\ & \textstyle
+ 2 \left( \frac{\lambda_t - \lambda_{t-1} }{2}  \right)^2
 \left( 
  ( \beta_{t-1} - \beta_t )
- 2 \left(  \frac{1+\beta_t-\lambda_{t-1}}{2}  \right) \left( \frac{\beta_{t-1} - \beta_t}{2}  \right) - \left( \frac{\beta_{t-1} - \beta_t}{2}  \right)^2
 \right)
\\ & \textstyle
  + 
2 \left( \beta_t - \beta_{t-1} \right)
\left( 1  - \frac{2 +\beta_{t-1} + \beta_{t} - 2 \lambda_{t-1} }{ 4 }   \right)
\left( \frac{1+\beta_{t} - \lambda_{t-1}}{ 2} + \frac{1+\beta_{t} - \lambda_{t}}{ 2}    \right) \left( \frac{\lambda_t - \lambda_{t-1}}{2}  \right)
\\ & \textstyle
 + 
\left( \beta_t - \beta_{t-1} \right)^2
\left( 1  - \frac{2 + \beta_{t-1} + \beta_{t} - 2 \lambda_{t-1} }{ 4 }   \right)^2.
\end{split}
\end{equation}
Let us analyze $\textcircled{1}$ first.
We have
\begin{equation}\label{cir1}
\begin{split}
 \textstyle \textcircled{1} & \textstyle =
\left(\frac{\lambda_t-\lambda_{t-1}}{2}\right)^4+2\left(\frac{\lambda_t-\lambda_{t-1}}{2}\right)^2\left(2\beta_t-\left(\frac{1+\beta_t-\lambda_t}{2}\right)^2-\left(\frac{1+\beta_t-\lambda_{t-1}}{2}\right)^2\right)  + \left(\frac{1+\beta_t-\lambda_t}{2}+\frac{1+\beta_t-\lambda_{t-1}}{2}\right)^2\left(\frac{\lambda_t-\lambda_{t-1}}{2}\right)^2
\\ &
\textstyle
=
\left(\frac{\lambda_t-\lambda_{t-1}}{2}\right)^2
\left(  \left(\frac{\lambda_t-\lambda_{t-1}}{2} \right)^2  
+2  \left(2\beta_t-\left(\frac{1+\beta_t-\lambda_t}{2}\right)^2-\left(\frac{1+\beta_t-\lambda_{t-1}}{2}\right)^2\right)
+ \left( \frac{1+\beta_t-\lambda_t}{2}+\frac{1+\beta_t-\lambda_{t-1}}{2} \right)^2
  \right)
\\ &  
\textstyle
=
\left(\frac{\lambda_t-\lambda_{t-1}}{2}\right)^2
\left( 4 \beta_t + \left(\frac{\lambda_t-\lambda_{t-1}}{2} \right)^2  
-\left(\frac{1+\beta_t-\lambda_t}{2}\right)^2-\left(\frac{1+\beta_t-\lambda_{t-1}}{2}\right)^2
+ 2 \left( \frac{1+\beta_t-\lambda_t}{2} \right) \left( \frac{1+\beta_t-\lambda_{t-1}}{2} \right)  \right)
\\ &  
\textstyle
=
\left(\frac{\lambda_t-\lambda_{t-1}}{2}\right)^2 4 \beta_t.
\end{split}
\end{equation}
Let us switch to $\textcircled{2}$. 
It is noted that all the terms in $\textcircled{2}$ have the factor $|\beta_{{t-1}} - \beta_{t}|$ and hence \textcircled{2} will disappear when the momentum parameter is set to a constant value. 
By using $|\lambda_{t} - \lambda_{{t-1}}|\leq 1 $, $|\beta_{{t-1}} - \beta_{t}| \leq 1$,
$\lambda_{t} \in [0,1]$, and $\beta_{t} \in [0,1]$, 
we have
\begin{equation} \label{e65}
\begin{split}
\textcircled{2} & \leq
8 (\beta_{t-1}-\beta_t)^2 
+ 8 | \beta_{t-1} - \beta_t | | \lambda_{t} - \lambda_{t-1} |  .
\end{split}
\end{equation}
Using (\ref{eqa}), (\ref{eqb}), and (\ref{cir1}),
we can bound $\|G_t\|_2$ as follows:
\begin{equation}
\begin{split}
\|G_t\|_2  & = \sqrt{\lambda_{\max}(G_tG_t^*)} = \sqrt{ a' + |b'| }
\\& 
= \sqrt{
\frac{4\beta_t -  (1+\beta_t-\lambda_t)(1+\beta_{t}-\lambda_{t-1})}{4\beta_t - (1+\beta_t - \lambda_t)^2} 
  + \frac{ 2(\beta_{t-1} - \beta_t) - (\beta_{t-1} - \beta_t ) ( 1 + \beta_t - \lambda_t) }{4\beta_t - (1+\beta_t - \lambda_t)^2}
+ \frac{
\sqrt{
\left( \frac{\lambda_t-\lambda_{t-1}}{2}\right)^2 4 \beta_t
+ \textcircled{2}
}
}{ 2\left( \beta_t - (\frac{ 1+\beta_t - \lambda_t}{2} )^2  \right)} }. 
\end{split}
\end{equation}
If $\beta_{{t}} = \beta_{{t-1}} = \sqrt{\beta}$, then the above reduces to
\begin{equation}
\|G_t\|_2 =\sqrt{\frac{4\beta -  (1+\beta-\lambda_t)(1+\beta-\lambda_{t-1})+2|\lambda_t-\lambda_{t-1}|\sqrt{\beta}}{4\beta - (1+\beta - \lambda_t)^2}},
\end{equation}
which can be simplified as
\begin{equation} \label{60}
\|G_t\|_2 = \sqrt{\Phi_t} :=
\begin{cases}
 \underbrace{ \sqrt{1+ \frac{\lambda_{t} - \lambda_{t-1}}{ (1+\sqrt{\beta})^2 - \lambda_{t}  } }  }_{ \text{extra factor} } \quad 
\text{if } \lambda_{t} \geq \lambda_{t-1} \\
 \underbrace{ \sqrt{1+ \frac{\lambda_{t-1} - \lambda_{t}}{  \lambda_{t}  - (1-\sqrt{\beta})^2  } }  }_{ \text{extra factor} } \quad
\text{if } \lambda_{t-1} \geq \lambda_{t},
\end{cases}
\end{equation}
where we denote $\Phi_{t} := 1+ \frac{\lambda_{t} - \lambda_{t-1}}{ (1+\sqrt{\beta})^2 - \lambda_{t}  } $ if $\lambda_{t} \geq \lambda_{t-1}$,
while denote $\Phi_{t} := 1+ \frac{\lambda_{t-1} - \lambda_{t}}{  \lambda_{t}  - (1-\sqrt{\beta})^2  }  $ if $\lambda_{t-1} \geq \lambda_{t}$.

On the other hand, for the case $\beta_{{t}} \neq \beta_{{t-1}}$, by using $\sqrt{y+z} \leq \sqrt{y} + \sqrt{z}$, we have
\begin{equation} \label{61}
\begin{split}
 \|G_t\|_2  
& \leq  
\sqrt{
\Phi_t
+ 
\frac{ 2(\beta_{t-1} - \beta_t) - (\beta_{t-1} - \beta_t ) ( 1 + \beta_t - \lambda_t) 
}{ | 4\beta_t - (1+\beta_t - \lambda_t)^2| }
+ 
\frac{ \sqrt{ \textcircled{2}} }{ 2\left( \beta_t - \left(\frac{ 1+\beta_t - \lambda_t}{2} \right)^2  \right)}}
\\ &
\overset{ \eqref{e65}}{\leq} 
\sqrt{ \Phi_t + 
\frac{ | \beta_{t-1} - \beta_t| ( 1 + 2 \sqrt{8} - \beta_t + \lambda_t)
+ 2 \sqrt{8} \sqrt{ | \beta_{t-1} - \beta_t | | \lambda_t - \lambda_{t-1}| }  
}{| 4\beta_t - (1+\beta_t - \lambda_t)^2  |} }
\\ &
\leq
\sqrt{\Phi_t} + 3 \sqrt{ \frac{ | \beta_{t-1} - \beta_t| }{ | 4\beta_t - (1+\beta_t - \lambda_t)^2  |} } +  3 \sqrt{  \frac{ \sqrt{ | \beta_{t-1} - \beta_t| | \lambda_t - \lambda_{t-1}|} }{ |4\beta_t - (1+\beta_t - \lambda_t)^2| }   }  .
\end{split}
\end{equation}
Since the spectral norm of a block diagonal matrix $M ={\rm diag}(M_1,M_2, \dots, M_d)$ satisfies
$\| M \|_2 = \max_{{i\in [d]}} { \| M_i \|_2 } $,
we can now conclude the following two cases.

When $\beta_{t}$ is a constant for all $t$, we have 
\begin{equation}
\begin{split}
\| P_t^{-1}P_{t-1} \|_2 & \leq \| {\rm Diag}(Q_{t,1},\dots,Q_{t,d})^{-1} {\rm Diag}(Q_{t-1,1},\dots,Q_{t-1,d}) \|_2
\\ &
\overset{\eqref{60}, \eqref{brev1}}{=} \max_{i \in [d]}
\begin{cases}
 \underbrace{ \sqrt{1+ \frac{\eta\lambda_{t,i} - \eta\lambda_{t-1,i}}{ (1+\sqrt{\beta_t})^2 - \eta\lambda_{t,i}  } }  }_{ \text{extra factor} } \quad 
\text{if } \eta\lambda_{t,i} \geq \eta\lambda_{t-1,i} \\
 \underbrace{ \sqrt{1+ \frac{\eta\lambda_{t-1,i} - \eta\lambda_{t,i}}{  \eta\lambda_{t,i}  - (1-\sqrt{\beta_t})^2  } }  }_{ \text{extra factor} } \quad
\text{if } \eta\lambda_{t-1,i} \geq \eta\lambda_{t,i}.
\end{cases}
\end{split}
\end{equation}
When $\beta_t$ is set adaptively, by
\eqref{61} and \eqref{brev1},
we have
\begin{equation}
\begin{split}
& \textstyle \| P_t^{-1}P_{t-1} \|_2  \leq \| {\rm Diag}(Q_{t,1},\dots,Q_{t,d})^{-1} {\rm Diag}(Q_{t-1,1},\dots,Q_{t-1,d}) \|_2
\\ & \textstyle
\leq \max_{i \in [d]}
\begin{cases}
 \underbrace{ \sqrt{1+ \frac{\eta\lambda_{t,i} - \eta\lambda_{t-1,i}}{ (1+\sqrt{\beta_t})^2 - \eta\lambda_{t,i}  } }
+ 3 \sqrt{ \frac{ | \beta_{t-1} - \beta_t| }{| 4\beta_t - (1+\beta_t - \eta\lambda_{t,i})^2| } } +  3  \sqrt{  \frac{ \sqrt{ | \beta_{t-1} - \beta_t| | \eta\lambda_{t,i} - \eta\lambda_{t-1,i}|} }{ |4\beta_t - (1+\beta_t - \eta\lambda_{t,i})^2| }   }
   }_{ \text{extra factor} , \text{ if } \lambda_{t,i} \geq \lambda_{t-1,i} },  
 \\
 \underbrace{ \sqrt{1+ \frac{\eta\lambda_{t-1,i} - \eta\lambda_{t,i}}{  \eta\lambda_{t,i}  - (1-\sqrt{\beta_t})^2  } } 
+ 3 \sqrt{ \frac{ | \beta_{t-1} - \beta_t| }{ |4\beta_t - (1+\beta_t - \eta\lambda_{t,i})^2 |} } +  3  \sqrt{  \frac{ \sqrt{ | \beta_{t-1} - \beta_t| | \eta\lambda_{t,i} - \eta\lambda_{t-1,i}| } }{| 4\beta_t - (1+\beta_t - \eta\lambda_{t,i})^2| }   }
  }_{ \text{extra factor} , \text{ if } \lambda_{t-1,i} \geq \lambda_{t,i} }
\end{cases}
\end{split}
,
\end{equation}

We now have completed the proof.

\end{proof}

\section{Proof of Lemma~\ref{lem:lambdadiff}} \label{app:lem:lambdadiff}

\textbf{Lemma~\ref{lem:lambdadiff}}
\textit{
Suppose that the Hessian of $f(\cdot)$ is $L_{H}$-Lipschitz,
i.e.,
$\| \nabla^2 f(x) - \nabla^2 f(y) \|_2 \leq  L_H \| x - y \|_2,$
for any pair of $x,y$.
Then, 
$| \lambda_{t,i} - \lambda_{{t-1},i} | \leq L_H \| w_t - w_{t-1}\|_2.$
}

\begin{proof}

\begin{equation}
\begin{aligned}
| \lambda_{t,i} - \lambda_{{t-1},i} | 
& \overset{(a)}{ \leq } \| H_t - H_{t-1}\|_{2}
\\ & = 
\| \eta \int_0^1 \nabla^2 f\big( (1-\tau) w_t + \tau w_* \big) d \tau 
-  \eta \int_0^1 \nabla^2 f\big( (1-\tau) w_{t-1} + \tau w_* \big) d \tau
\|_2
\\ &\overset{(b)}{\leq}  \eta L_H \| w_t - w_{t-1}\|_2,
\end{aligned}
\end{equation}
where (a) is due to Theorem~\ref{thm:p} below and (b) is by the assumption that the Hessian is $L_{H}$-Lipschitz.
\end{proof}

\begin{theorem}[Theorem 8.1 in \cite{B07}] \label{thm:p}
Let $A$ and $B$ be Hermitian matrices with eigenvalues 
$\lambda_{1}(A) \geq \lambda_{2}(A) \geq \dots \geq \lambda_{d}(A)$
and 
$\lambda_{1}(B) \geq \lambda_{2}(B) \geq \dots \geq \lambda_{d}(B)$.
Then, 
\[
\max_{j \in [d] } | \lambda_j(A) - \lambda_j(B)	| \leq \| A - B \|_2.
\]
\end{theorem}

\section{Proof of Theorem~\ref{thm:PL}} \label{app:thm:PL}

\textbf{Theorem~\ref{thm:PL}}
\textit{
Let
$\theta = 2 \left( \frac{L}{4} \left( 1 + \frac{1}{c_{\eta}}\right) - c_{\mu} \mu    \right) > 0$ for any $c_{{\eta}} \in (0,1]$ 
and any $c_{{\mu}} \in (0,\frac{1}{4}]$.
Set the step size $\eta = \frac{c_{\eta} }{L}$ and set the momentum parameter $\beta_{t}$ so that for all~$t$,
$\beta_t \leq \sqrt{ 
\left(1 - \tilde{c} c_{\eta}^2 \frac{\mu}{L}  \right) 
\left( 1 - \frac{c_{\mu} \mu }{  \frac{L}{4} \left( 1 + \frac{1}{c_{\eta}}\right) + \frac{\theta}{2} }
\right)}$
for some constant $\tilde{c} \in (0,1]$.
Then, HB has
\begin{equation} 
V_t
 \leq \left( 1 - \tilde{c} c_{\eta}^2 \frac{ \mu }{L }  \right)^t V_0
 = \left( 1 - \Theta\left( \frac{ \mu }{L} \right)  \right)^t V_0,
\end{equation}
where the Lyapunov function $V_{t}$ is defined on \eqref{lyp}.
}

\begin{proof}
From the update of HB, we have following three basic equalities and inequalities:
\begin{align}
w_t - w_{t-1} & = - \eta \nabla f(w_{t-1}) + \beta_{t-1}  (w_{t-1} - w_{t-2})
\\ \| w_t - w_{t-1} \|^{2} & = \eta^2 \| \nabla f(w_{t-1}) \|^2 + \beta_{t-1}^2 \|    w_{t-1} - w_{t-2}  \|^2 - 2 \eta \beta_{t-1} \langle \nabla f(w_{t-1}),  w_{t-1} - w_{t-2} \rangle . \label{K1}
\\ f(w_t) & \leq f(w_{t-1}) + \langle \nabla f(w_{t-1}), w_{t} - w_{{t-1}} \rangle + \frac{L}{2} \| w_t - w_{t-1}\|^2 .
\end{align}
Combining the above three, we get
\begin{equation} \label{K2}
\begin{split}
f(w_t) & \leq f(w_{t-1}) - \eta \| \nabla f(w_{t-1}) \|^2 
+ \beta_{t-1}  \langle \nabla f(w_{t-1}),   w_{t-1} - w_{t-2} \rangle
+ \frac{\eta^2 L}{2} \| \nabla f(w_{t-1}) \|^2
\\ & + \frac{ L \beta_{t-1}^2 }{2} \|    w_{t-1} - w_{t-2}  \|^2 
- L \eta \beta_{t-1} \langle \nabla f(w_{t-1}),   w_{t-1} - w_{t-2} \rangle.
\end{split}
\end{equation}
Multiplying (\ref{K1}) by $\theta$ and adding it to (\ref{K2}), we have
\begin{equation}
\begin{split}
f(w_t) + \theta \| w_t - w_{t-1} \|^2
& \leq f(w_{t-1}) - \eta \| \nabla f(w_{t-1}) \|^2 
+  \beta_{t-1} \langle \nabla f(w_{t-1}),    w_{t-1} - w_{t-2}  \rangle
+ \frac{\eta^2 L}{2} \| \nabla f(w_{t-1}) \|^2
\\ & + \frac{ L \beta_{t-1}^2 }{2} \|   w_{t-1} - w_{t-2}  \|^2 
- L \eta \beta_{t-1} \langle \nabla f(w_{t-1}),   w_{t-1} - w_{t-2} \rangle
\\ &
+ \theta \left( \eta^2 \| \nabla f(w_{t-1}) \|^2 +  \beta_{t-1}^2 \|    w_{t-1} - w_{t-2}  \|^2 - 2 \eta \beta_{t-1}  \langle \nabla f(w_{t-1}),  w_{t-1} - w_{t-2} \rangle    \right).
\end{split}
\end{equation}
After grouping some terms, we obtain
\begin{equation} \label{79}
\begin{split}
f(w_t) + \theta \| w_t - w_{t-1} \|^2
& \leq f(w_{t-1})
+ \left(   \frac{L}{2}  + \theta \right) \beta_{t-1}^2 \|   w_{t-1} - w_{t-2}\|^2
\\ & \quad + \left( \frac{\eta^2 L}{2}  - \eta + \theta \eta^2 \right) \| \nabla f(w_{t-1}) \|^2
\\ & \quad 
+ \left(  1- L \eta - \theta 2 \eta  \right) \beta_{t-1}  \langle \nabla f(w_{t-1}),  w_{t-1} - w_{t-2} \rangle,
\end{split}
\end{equation}
which can be further bounded as
\begin{equation} \label{1}
\begin{split}
f(w_t) + \theta \| w_t - w_{t-1} \|^2
& \leq f(w_{t-1})
+ (   \frac{L}{2}  + \theta ) \beta_{t-1}^2 \|   w_{t-1} - w_{t-2}\|^2
+ ( \frac{\eta^2 L}{2}  - \eta + \theta \eta^2 ) \| \nabla f(w_{t-1}) \|^2
\\ &  \quad 
+ \frac{ c_{\eta} \left(  1- L \eta - \theta 2 \eta  \right) }{L} \| \nabla f(w_{t-1}) \|^2 + \frac{L}{4 c_{\eta}} \left(  1- L \eta - \theta 2 \eta  \right)  \beta_{t-1}^2 \|   w_{t-1} - w_{t-2}\|^2
\\ &  =
f(w_{t-1}) - \frac{ c_{\eta}^2 }{L} \left( \frac{1}{2} + \frac{\theta}{L} \right)
 \| \nabla f(w_{t-1})\|^2 + \left( \frac{L}{4} \left(1+ \frac{1}{c_{\eta}}\right) + \frac{\theta}{2}  \right)
\beta_{t-1}^2 \|   w_{t-1} - w_{t-2}\|^2,
\end{split}
\end{equation}
where we used $\langle a, b \rangle \leq \|a\| \|b\| \leq \frac{1}{2} \| a \|^{2} + \frac{1}{2} \| b\|^{2}$ with $a \leftarrow \sqrt{\frac{2 c_{\eta}}{L} }  \nabla f(w_{t-1}) $ and $b \leftarrow \beta_{t-1} \sqrt{ \frac{L}{2c_{\eta}} } ( w_{t-1} - w_{t-2} )$ to bound the last term on \eqref{79},
and we get the equality by grouping the common terms and replacing $\eta = \frac{c_{\eta}}{L}$.

Now subtracting $f(w_*)$ from both sides of \eqref{1} and using the PL inequality, i.e., $2 \mu ( f(w_{t-1}) - f(w_*) ) \leq \| \nabla f(w_{t-1}) \|^{2}$, we get 
\begin{equation} \label{81}
f(w_t) - f(w_*) + \theta \| w_t - w_{t-1} \|^2
\leq \left(  1 - \frac{c_{\eta}^2  \mu}{L} \left( 1 + \frac{2 \theta}{L} \right)   \right) \left( f(w_{t-1}) - f(w_*) \right) +
\left( \frac{L}{4} \left(1+ \frac{1}{c_{\eta}}\right) + \frac{\theta}{2}  \right)
\beta_{t-1}^2 \|   w_{t-1} - w_{t-2}\|^2.
\end{equation}
To complete the proof, we need to bound the r.h.s.~of inequality \eqref{81} by
\[\left(1 - \tilde{c} c_{\eta}^2 \frac{ \mu }{L }   \right)
\left( f(w_{t-1}) - f(w_*) + \theta \| w_{t-1} - w_{t-2} \|^2
  \right),
\]
for some $\tilde{c} \in (0,1]$.
Since $-(1+ \frac{2 \theta}{L} ) \leq - \tilde{c}$, we only need
$\beta_{t-1}^2 \leq 
(1 - \tilde{c} c_{\eta}^2 \frac{ \mu }{L }   ) 
\left( 
\frac{ \theta }{ \frac{L}{4} \left(1+ \frac{1}{c_{\eta}}\right) + \frac{\theta}{2} } \right) 
= \left( 1 - \tilde{c} c_{\eta}^2 \frac{ \mu }{L }   \right) 
\left( 1 - \frac{c_{\mu} \mu }{  \frac{L}{4} \left( 1 + \frac{1}{c_{\eta}}\right) + \frac{\theta}{2} } \right)
$.

\end{proof}


\section{Proof of Theorem~\ref{thm:meta}} \label{app:thm:meta}

\textbf{Theorem~\ref{thm:meta}}
\textit{
Suppose assumption $\clubsuit$ holds
and that \avg{\lambda_{\min}(H_t)}{w_*} 
and co-diagonalization $\spadesuit$ hold for all $t$.
Let
$\theta := 2 \left( \frac{L}{4} \left( 1 + \frac{1}{c_{\eta}}\right) - c_{\mu} \mu    \right) > 0$, where $c_{{\mu}} \in (0,\frac{1}{4}]$.
Set the step size $\eta = \frac{c_{\eta} }{L}$
for any constant $c_{\eta} \in (0,1]$
and set the momentum parameter 
$\beta_t = \left(1- c \sqrt{  \eta \lambda_{\min}(H_t) } \right)^2 $ for some $c \in (0,1)$
satisfying 
$\beta_t \leq
\sqrt{
\left(1 - \tilde{c} c_{\eta}^2 \frac{\mu}{L}  \right) 
\left( 1 - \frac{c_{\mu} \mu }{  \frac{L}{4} \left( 1 + \frac{1}{c_{\eta}}\right) + \frac{\theta}{2} }
\right) }$
for some constant $\tilde{c} \in (0,1]$.
Then, for all $t$, the iterate $w_{t}$ of HB satisfies \eqref{eq:V}, i.e.,
the Lyapunov function $V_{t}$ decays linearly for all $t$.
Furthermore,
there exists a time 
$t_{0} = \tilde{\Theta}\left( \frac{1}{\tilde{c}} \frac{L}{\mu} \right) $ 
such that for all $t \geq t_{0}$, the instantaneous rate $\| \Psi_t \|_2$ at $t$ is
\begin{equation} 
\| \Psi_t \|_2 = 1 - \frac{c \sqrt{c_{\eta}}}{2} \frac{1}{\sqrt{\kappa_t}}
= 1 - \Theta\left( \frac{1}{\sqrt{\kappa_t} }  \right), 
\end{equation}
where $\kappa_t:= \frac{L}{\lambda_{\min}(H_t)}$.
Consequently, 
\[
\| w_{T+1} - w_* \| = O\left( \prod_{t=t_0}^{T} \left(  1 - \Theta\left(\frac{1}{ \sqrt{\kappa_{t}}} \right)  \right) \right) \| w_{t_0} - w_* \|.
\]
}

\begin{proof}
Let us first recall that 
$\{ \lambda_{{t,i}} \}_{{i=1}}^{d}$ are the eigenvalues of $H_t := \int_0^1 \nabla^2 f\big( (1-\tau) w_t + \tau w_* \big) d \tau
$ in the decreasing order. 
Hence, we can write $\beta_t = (1-  c \sqrt{\eta \lambda_{t,d}} )^{2}$
for some $c \in (0,1)$, where $\lambda_{{t,d}} = \lambda_{{\min}}(H_t) \geq \lambda_{*}> 0$ for some constant $\lambda_{*}>0$.

From Lemma~\ref{vb}, we know the instantaneous rate $\|\Psi_t\|_{2}$ satisfies
\begin{align} \label{ext:0}
\|\Psi_t\|_2&=
\begin{cases}
\left( \sqrt{\beta_t} \right) \times
\underbrace{ \left( \max_{i \in [d]} \sqrt{1+ \frac{\eta \lambda_{t,i} -\eta \lambda_{t-1,i}}{ (1+\sqrt{\beta_t})^2 - \eta \lambda_{t,i}  } }
 + \mathbbm{1}\{\beta_t \neq \beta_{t-1}\} \phi_{t,i}
\right) }_{\text{extra factor 1} },
\text{ if } \lambda_{t,i} \geq \lambda_{t-1,i} \\
\left( \sqrt{\beta_t} \right) 
\times
\underbrace{ \left( \max_{i \in [d]}
\sqrt{1+ \frac{\eta \lambda_{t-1,i} - \eta \lambda_{t,i}}{ \eta \lambda_{t,i}  - (1-\sqrt{\beta_t })^2  } } 
 + \mathbbm{1}\{\beta_t \neq \beta_{t-1}\} \phi_{t,i}
\right) }_{\text{extra factor 2} }, 
\text{ if } \lambda_{t-1,i} \geq \lambda_{t,i}. \\
\end{cases}
\end{align}

Let us first analyze the denominator of the first term in the extra factors.
There are two cases. 
\begin{itemize}
\item \textbf{Suppose $\lambda_{t,i} \geq \lambda_{t-1,i}$ (the first term in the extra factor is $\sqrt{1+ \frac{\eta \lambda_{t,i}-\eta \lambda_{t-1,i}}{\left(1+\sqrt{\beta_t}\right)^2 - \eta \lambda_{t,i}} }$)
:}\\
We first show that the denominator $\left(1+\sqrt{\beta_t}\right)^2 - \eta \lambda_{t,i}$ is non-zero. If it were zero, then it implies that
\begin{equation}
\begin{split}
\left(1+\sqrt{\beta_t}\right) =  \sqrt{ \eta \lambda_{t,i} },
\end{split}
\end{equation}
which would contradict to the choice of $\beta_{t} = (1-  c \sqrt{\eta \lambda_{t,d}} )^{2} > (1 - \sqrt{\eta \lambda_{t,i} })^2$ for all $i \in [d]$.
It suffices to assume that the denominator is lower bounded by a positive number, i.e., 
\begin{equation} \label{eq:l1}
\left(1+\sqrt{\beta_t}\right)^2 - \eta \lambda_{t,i} \geq c_{0} > 0.
\end{equation}
Otherwise, given that $\lambda_{t,i} \geq \lambda_{t-1,i}$, we can trivially upper-bound $\sqrt{1+ \frac{\eta \lambda_{t,i}-\eta \lambda_{t-1,i}}{\left(1+\sqrt{\beta_t}\right)^2 - \eta \lambda_{t,i}} }$ by $1$.
\item \textbf{Suppose $\lambda_{t,i} \leq \lambda_{t-1,i}$ (the first term in the extra factor in $\sqrt{1+\frac{\eta \lambda_{t-1,i} - \eta \lambda_{t,i} }{ \eta \lambda_{t,i} - \left(1- \sqrt{\beta_t }\right)^2 }}  $):}\\
We need to lower-bound $ \eta \lambda_{t,i} - \left(1-\sqrt{\beta_t}\right)^2 $.
We have
\begin{equation} \label{eq:l2}
\begin{aligned}
\eta \lambda_{t,i} - \left(1-\sqrt{\beta_t }\right)^2  = \eta \lambda_{t,i} - ( 1 - (1 - c \sqrt{\eta \lambda_{t,d}} ) )^2 
= 
\eta \lambda_{t,i} - c^2 \eta \lambda_{t,d} \geq \eta \lambda_{t,d} (1-c^2)  := c_1 > 0.
\end{aligned}
\end{equation}
\end{itemize}
So the first term in the extra factor can be bounded as
\begin{equation}
\max\left( \max_{i \in [d]} \sqrt{1+ \frac{\eta \lambda_{t,i} - \eta \lambda_{t-1,i}}{ (1+\sqrt{\beta_t})^2 - \eta \lambda_{t,i}  }} ,
\max_{i \in [d]} \sqrt{1+ \frac{\eta \lambda_{t-1,i} - \eta \lambda_{t,i}}{ \eta \lambda_{t,i}  - (1-\sqrt{\beta_t})^2  } }
\right)
\leq \max_{i \in [d]} \sqrt{ 1 + \frac{ \eta | \lambda_{t-1,i} - \lambda_{t,i}|   }{\min(c_0,c_1)}  } .
\end{equation}
Using Lemma~\ref{lem:lambdadiff}, we further have 
\begin{equation} \label{ext:1}
\max\left( \max_{i \in [d]} \sqrt{1+ \frac{\eta \lambda_{t,i} - \eta \lambda_{t-1,i}}{ (1+\sqrt{\beta_t})^2 - \eta \lambda_{t,i}  }} ,
\max_{i \in [d]} \sqrt{1+ \frac{\eta \lambda_{t-1,i} - \eta \lambda_{t,i}}{ \eta \lambda_{t,i}  - (1-\sqrt{\beta_t})^2  } }
\right)
\leq 
\sqrt{ 1 + \frac{ \eta L_H \| w_{t-1} - w_t \|_2  }{\min(c_0,c_1)}  }.
\end{equation}
It is noted that if
a constant value of the $\beta$ is used, e.g., $\beta_t = (1-  c \sqrt{\eta \lambda_{*}} )^{2} = (1- \frac{c\sqrt{c_{\eta}}}{ \sqrt{\kappa}} )^2  $ for some $c \in (0,1)$, 
where $\kappa:= \frac{L}{\lambda_*}$,
then \eqref{ext:1} still holds for some $c_{0},c_{1}>0$, which can be seen by tracing the derivations in \eqref{eq:l1} and \eqref{eq:l2}.

Now let us switch to bounding the second term of the extra factors.
The second term is zero if the momentum parameter is set to a constant value during the iterations, e.g., $\beta_t = (1-  c \sqrt{\eta \lambda_{*}} )^{2} = (1- \frac{c\sqrt{c_{\eta}}}{ \sqrt{\kappa}} )^2  $.
If the momentum parameter is set adaptively as
$\beta_t = (1-  c \sqrt{\eta \lambda_{t,d}} )^{2}$ 
for some $c \in (0,1)$,
then from Lemma~\ref{lem:lambdadiff} \eqref{phidef} we need to bound the r.h.s. of the following,
\begin{equation} \label{phi}
\phi_{t,i}\leq 3 \sqrt{ \frac{ | \beta_{t-1} - \beta_t| }{ |4\beta_t - (1+\beta_t - \eta\lambda_{t,i})^2| } } +  3  \sqrt{  \frac{ \sqrt{ | \beta_{t-1} - \beta_t| | \eta\lambda_{t,i} - \eta\lambda_{t-1,i}|} }{ |4\beta_t - (1+\beta_t - \eta\lambda_{t,i})^2 |
}}. 
\end{equation}
From \eqref{val} and the proof of Lemma~\ref{lem:lambdadiff} (i.e., \eqref{b} - \eqref{61}), the term $| 4\beta_t - (1+\beta_t - \eta\lambda_{t,i})^2 |$ in r.h.s. of \eqref{phi} is actually $2 \left( \Im{z_{t,i}} \right)^2$,
i.e.,  $| 4\beta_t - (1+\beta_t - \eta\lambda_{t,i})^2 | = 2 \left( \Im{z_{t,i}} \right)^2$,
where $z_{t,i}$, $\bar{z}_{t,i}$ are eigenvalues of
$\tilde{\Sigma}_{t,i}:= \begin{bmatrix}
1+\beta_t-\lambda_i& -\beta_t\\1& 0
\end{bmatrix}$.
When $\beta_t > \left( 1 - \sqrt{\eta \lambda_{t,d} }  \right)^2 $, the sub-matrix $\tilde{\Sigma}_{i}$ has complex eigenvalues and hence 
$(\Im{z_i})^2 > 0$.
Since $\beta_t = (1-  c \sqrt{\eta \lambda_{t,d}} )^{2}$ 
for some $c \in (0,1)$, we know
\begin{equation} \label{eq:33}
| 4\beta_t - (1+\beta_t - \eta\lambda_{t,i})^2 | \geq c_3^2 > 0
\end{equation} 
for some constant $c_{3} > 0$. On the other hand, the factor
$| \beta_{t-1} - \beta_t|$ in r.h.s. of \eqref{phi} can be bounded as
\begin{equation} \label{ext:3}
\begin{split}
| \beta_{t-1} - \beta_t| & = 
\left| \left( 1 - c \sqrt{\eta \lambda_{t-1,d} }  \right)^2 
- \left( 1 - c \sqrt{\eta \lambda_{t,d} }  \right)^2  \right| 
\\ &
= \left| c \left(  \sqrt{\eta \lambda_{t,d}} - \sqrt{\eta \lambda_{t-1,d} } \right) 
\left(  2  - c \left(  \sqrt{\eta \lambda_{t,d}} + \sqrt{\eta \lambda_{t-1,d} }  \right)  \right) \right|
\\ &
\leq c \sqrt{ |\eta \lambda_{t,d} - \eta \lambda_{t-1,d}| }  
\left| \left(  2  - c \left(  \sqrt{\eta \lambda_{t,d}} + \sqrt{\eta \lambda_{t-1,d} }  \right)  \right) \right|
\\ & \leq 2 c \sqrt{ \eta L_H \| w_t - w_{t-1 }\|_2 },
\end{split}
\end{equation}
where the last inequality uses Lemma~\ref{lem:lambdadiff}.

Hence by \eqref{ext:0}, (\ref{ext:1})-(\ref{ext:3}), we get
\begin{equation} \label{ext:2}
\begin{split}
& \|\Psi_t \|_2   \leq \\ & \sqrt{\beta_t } \left(  \sqrt{ 1 + \frac{ \eta L_H \| w_{t-1} - w_t \|  }{\min(c_0,c_1)}  }  + \mathbbm{1}
 [\beta_t \neq \beta_{t-1} ] 
\left( \frac{3}{c_3} \left( \sqrt{2c \sqrt{ \eta L_H \| w_t -w_{t-1}\|_2 } } + \sqrt{ \sqrt{2c} \sqrt{\eta L_H \| w_t - w_{t-1}\|_2} }   \right)    \right)
  \right).
\end{split}
\end{equation}
The strategy now is to show the following two items hold for all iterations $t \geq t_{0}$
for some $t_{0}$.
\begin{itemize}
\item \textbf{(first item)}
$\sqrt{ 1 + \frac{ \eta L_H \| w_{t-1} - w_t \|  }{\min(c_0,c_1)}  } 
\leq 1 + \frac{c}{4} \sqrt{\eta \lambda_{\min}(H_t) } $
\item \textbf{(second item)}
$\left( \frac{3}{c_3} \left( \sqrt{2c \sqrt{ \eta L_H \| w_t -w_{t-1}\|_2 } } + \sqrt{ \sqrt{2c} \sqrt{ \eta L_H \| w_t - w_{t-1}\|_2} }   \right)    \right)
\leq \frac{c}{4} \sqrt{\eta \lambda_{\min}(H_t) }   $.
\end{itemize}
The above items would allow us to show that
\begin{equation} \label{ext:2}
\|\Psi_t \|_2 
\leq 
\left(1 - c \sqrt{\eta \lambda_{\min}(H_t) } \right)
\left(  1 + \frac{c}{4} \sqrt{\eta \lambda_{\min}(H_t) } + \frac{c}{4} \sqrt{\eta \lambda_{\min}(H_t) } \right)
\leq 1 - \frac{c}{2} \sqrt{\eta \lambda_{\min}(H_t) }  =
1 - \Theta \left( \sqrt{ \frac{ \lambda_{\min}(H_t) }{ L } } \right),
\end{equation}
which is an accelerated linear rate.

So now let us analyze the number of iterations required for the first item to hold.
We have
\begin{equation}
\begin{split}
\log \left( \sqrt{ 1 + \frac{ \eta L_H \| w_{t-1} - w_t \|  }{\min(c_0,c_1)}  } \right)
& \leq 
\frac{1}{2} \log  \left(  1 + \frac{ \eta L_H \| w_{t-1} - w_t \|  }{\min(c_0,c_1)} \right)
\\ & \leq
\frac{ \eta L_H \| w_{t-1} - w_t \|  }{2 \min(c_0,c_1)}
\\ & \overset{ (a) }{ \leq }
\frac{ \eta L_H \sqrt{ \left( 1 - \tilde{c} c_{\eta}^2  \frac{\mu}{L}  \right)^t V_0} }{ 2 \min(c_0,c_1) } \sqrt{ \frac{1}{\theta}  }
\\ & \overset{ (b) }{ \leq } \log \left( 1 + \frac{c}{4} \sqrt{\eta \lambda_{\min}(H_t) }  \right),
\end{split}
\end{equation}
where (a) is by Theorem~\ref{thm:PL} and (b) holds when $t = \Theta\left( \frac{1}{\tilde{c} c_{\eta}^2 } \frac{L}{\mu} \log \left( \frac{L_H^2 V_0}{ \min\{c_0^2,c_1^2 \} \theta }   \right)  \right)
 $ iterations.
For the first term in the second item, we have
\begin{equation}
\begin{split}
\frac{3}{c_3} \left( \sqrt{2c \sqrt{ \eta L_H \| w_t -w_{t-1}\|_2 } }  \right) 
\leq \frac{c}{8} \sqrt{\eta \lambda_{\min}(H_t) }  
\iff \| w_t - w_{t-1} \|_2 \leq \left( \frac{c c_3 \sqrt{\eta \lambda_{\min}(H_t) } }{24 \sqrt{2 c} (\eta L_H)^{1/4} } \right)^4. 
\end{split}
\end{equation}
Using Theorem~\ref{thm:PL} again, we have 
\begin{equation}
\| w_t - w_{t-1} \|_2 \leq \sqrt{ \left( 1 - \tilde{c} c_{\eta}^2  \frac{\mu}{L}  \right)^t V_0 }
\sqrt{ \frac{ 1 }{ \theta } }
\leq \left( \frac{c c_3 \sqrt{\eta \lambda_{\min}(H_t) } }{24 \sqrt{2 c} (\eta L_H)^{1/4} } \right)^4,
\end{equation}
where the last inequality holds when $t= \Theta \left( \frac{1}{\tilde{c} c_{\eta}^2 } \frac{L}{\mu} \log \left(\frac{ L^2 L_H^2  V_0 }{ c_3^8 \lambda_{*}^4 \theta }  \right) \right)$.
For the second term in the second item, we have
\begin{equation}
\frac{3}{c_3} \sqrt{ \sqrt{2c} \sqrt{\eta L_H \| w_t - w_{t-1}\|_2} }
\leq \frac{c}{8} \sqrt{\eta \lambda_{\min}(H_t) }
\iff  
\| w_t - w_{t-1} \|_2 \leq 
\left( \frac{c c_3 \sqrt{\eta \lambda_{\min}(H_t) } }{24 (2 c)^{1/4} (\eta L_H)^{1/4} } \right)^4.
\end{equation}
Using Theorem~\ref{thm:PL} again, we have 
\begin{equation}
\| w_t - w_{t-1} \|_2 \leq \sqrt{ \left( 1 - \tilde{c} c_{\eta}^2  \frac{\mu}{L}  \right)^t V_0 }
\sqrt{ \frac{ 1 }{ \theta } }
\leq \left( \frac{c c_3 \sqrt{\eta \lambda_{\min}(H_t) } }{24 (2 c)^{1/4} (\eta L_H)^{1/4} } \right)^4,
\end{equation}
where the last inequality holds when $t= \Theta \left( \frac{1}{\tilde{c} c_{\eta}^2 } \frac{L}{\mu} \log \left( \frac{L^2 L_H^2 V_0 }{ c_3^8 \lambda_{*}^4 \theta} \right) \right)$.

So now we can conclude that after $t_{0} = \tilde{\Theta}\left(  \frac{L}{\mu} \right)$ number of iterations, 
where $\tilde{\Theta}\left( \cdot  \right)$ hides the following logarithmic factor
\begin{equation} \label{103}
\max \left\{
\log \left( \frac{L_H^2 V_0}{ \min\{c_0^2,c_1^2 \} \theta }   \right),
\log \left(\frac{ L^2 L_H^2  V_0 }{ c_3^8 \lambda_{*}^4 \theta }  \right) \right \},
\end{equation}  
and $c_0, c_1, c_3 > 0$ are constants defined in \eqref{eq:l1}, \eqref{eq:l2}, \eqref{eq:33},
the instantaneous rate satisfies
$\| \Phi_t \|_2 \leq  1  - \frac{c}{2} \sqrt{\eta \lambda_{\min}(H_t) }
= 1 - \frac{c \sqrt{c_{\eta}}}{2} \frac{1}{\sqrt{\kappa_t}} =  1 - \Theta\left( \frac{1}{\sqrt{\kappa_t} }  \right)
$.

\end{proof}

\section{Proof of Lemma~\ref{lem:syn} and Lemma~\ref{lem:broad} } \label{app:lem:syn}

\textbf{Lemma~\ref{lem:syn}}
\textit{
Let $\alpha \in (1,4.34]$.
The function $f(w)$ in \eqref{obj:syn} is twice differentiable and has $L=2+2\alpha$-Lipschitz, $L_H= 4 \alpha$-Lipschitz Hessian, and its global optimal point is $w_{*}=0$.
Furthermore, it is non-convex but satisfies $\mu$-PL 
with $\mu =\frac{ \left(2+ \frac{sin(2w)}{w} \alpha\right)^2}{2+2\alpha} \geq \frac{ \left( 2- 0.46 \alpha \right)^2}{2+2 \alpha}$ at any $w$ and the average Hessian is $H_f(w) = 2 + \frac{sin(2w)}{w} \alpha > 0$, where $|\frac{sin(2w)}{w}| \leq 0.46$. 
Therefore, the condition \avg{\lambda_*}{w_*} holds with $\lambda_{*} = 2-0.46 \alpha$. 
}

\begin{proof}
Consider
\begin{equation}
f(w) = w^2 + \alpha\sin^2(w),
\end{equation}
for some constant $\alpha  \in (1,4.34]$, whose minimizer is $w_*=0$.
Its gradient and Hessian are
\begin{equation}
\begin{split}
\nabla f(w) &= 2w + 2\alpha \sin(w) \cos(w)
\\ H(w) &= 2 + 2 \alpha (\cos^2(w) - \sin^2(w) ).
\end{split}
\end{equation}
It is clear that $H(w)<0$ at some $w$ for any $\alpha >1$.
Also, for any two points $x,y$, we have $| H(x)- H(y) | =| 2 \alpha ( \cos(2x) - \cos(2y) )| \leq 4 \alpha | x- y|$, where we used the mean value theorem.
So the function has $L_{H}= 4 \alpha$-Lipschitz Hessian.

Now consider the average Hessian $H_f(w):=\int_{0}^{1} H( \theta w + (1-\theta) w_*) d\theta$. We have
\begin{equation}
\begin{split}
\int_{0}^{1} H( \theta w + (1-\theta ) w_*) d\theta
& = \int_0^1 \left( 2 + 2 \alpha (\cos^2(\theta w) - \sin^2(\theta w) )  \right) d\theta
= 2 + 2 \alpha \int_0^1 (\cos^2(\theta w) - \sin^2(\theta w) ) d\theta
\\ &
= 2 + 2 \alpha \int_0^1 \cos(2\theta w) d\theta
\\ &
= 2 + \frac{\alpha}{w} \sin(2\theta w) |_0^1 
= 2 + \frac{\alpha \sin(2w)}{w}  \overset{Lemma~\ref{lem:sin}}{ >} 2 - 0.46 \alpha > 0.  
\end{split}
\end{equation}

Also, since the Hessian $H(w)$ can be bounded by $2+2\alpha$, we know $f(\cdot)$ is $2+2\alpha$-smooth.

By Lemma~\ref{lem:AOPL}, we have
$f(w)$ is $\mu-$PL at w, where $\mu= \frac{ (2 + \frac{\alpha \sin(2w)}{w} )^2 }{2+2 \alpha} \geq \frac{(2-0.46\alpha)^2}{2+2\alpha}$.
We now have completed the proof.
\end{proof}

\begin{lemma} \label{lem:sin}
We have
$g(x):=\frac{\sin(2x)}{x} \geq  -0.46 $.
\end{lemma}
\begin{proof}
We have 
$g'(x)=0$ when $2x = \tan(2x)$.
Denote $x_{*}$ a critical point so that it satisfies $2x_* = \tan(2x_*)$.
Then 
\begin{equation}
|g(x_*)| = \left|\frac{ \sin (\arctan(2 x_*) ) }{x_*}\right|= \left|\frac{2}{ \sqrt{ (2x_*)^2 + 1}}\right|,
\end{equation}
where we used that
$\sin (\arctan(2 x) ) = \pm \frac{2x}{\sqrt{(2x)^2+1}}$. We hence know that the absolute value of $g(x)$ at a critical point is decreasing with $|x|$. We know $g(0)=2$ is the maximal value. With the computer-aided analysis (as there is no simple closed form), the second extreme value can be approximately lower bounded by $g(\cdot) \geq -0.46$.

\end{proof}

\textbf{Lemma~\ref{lem:broad}}
\textit{
Consider solving
\[
\textstyle
\min_{w \in \reals^d} F(w) := \min_{w \in \reals^d} f(w) + g(w).
\]
Suppose the function $F(w)$ satisfies assumption $\clubsuit$.
Denote $w_{*}$ a global minimizer of $F(w)$.
Assume that $f(w)$ is $2 \lambda_{*}$-strongly convex for some $\lambda_{*}>0$. If the average Hessian of the possibly non-convex $g(w)$ between $w$ and $w_{*}$ satisfies $\lambda_{{\min}}(H_g(w)) \geq - \lambda_{*}$. Then, $F(\cdot)$ satisfies \avg{\lambda_*}{w_*}.
}

\begin{proof}
The average Hessian of $F(w)$ between $w$ and $w_{*}$ is
\begin{equation}
\begin{split}
\int_{0}^{1} H_F( \theta w + (1-\theta ) w_*) d\theta
& = \int_0^1 H_f( \theta w + (1-\theta ) w_*) d\theta 
+ \int_0^1 H_g( \theta w + (1-\theta ) w_*) d\theta
\\ & \geq 2 \lambda_* +  \int_0^1 H_g( \theta w + (1-\theta ) w_*) d\theta
\\ & \geq \lambda_*.
\end{split}
\end{equation}

\end{proof}

\section{Example 3} \label{app:ex3}

\subsection{Proof of Lemma~\ref{lem:deep}} \label{app:lem:deep}

\textbf{Lemma~\ref{lem:deep}}
\textit{
($N=2$).
The function $f(u)$ in (\ref{deep}) is twice differentiable.
It is $L=3R^{2}$-smooth and has $L_{H}=6R$-Lipschitz Hessian
in the ball $\{ u: \|u\|_{\infty} \leq R\}$ for any finite $R>0$.
Furthermore, $f(u)$ is $\mu$-PL with $\mu = 2 \min_{i\in [d]} (u_i)^2$ 
at $u$, and the average Hessian $H_{f}(u)$ between $u$ and $\hat{u} \in \mathcal{U}$
satisfies $\lambda_{{\min}}( H_f(u) ) = \min_{i \in [d]} \left( (u_i)^{2} + u_i \hat{u}_i \right)$.
}

\begin{proof}

Recall the objective is
$f(u) = \frac{1}{ N^2} \| u^{\odot N} - w_* \|^2$.
The gradient is
$\nabla f(u) = \frac{2}{N} \left( u^{\odot N} - w_* \right) \odot u^{\odot N-1}$
and the Hessian is
$\nabla^2 f(u) = \text{Diag}\left( 2\frac{2N-1}{N}  u^{\odot 2 N -2} - 2\frac{N-1}{N} u^{\odot N-2} \odot w_* \right).$
We have
\begin{equation}
\begin{split}
\| \nabla f(u) \|^2  &= [ 
\frac{2}{N} \left( u^{\odot N} - w_* \right) \odot u^{\odot N-1}
 ]^\top  [ \frac{2}{N} \left( u^{\odot N} - w_* \right) \odot u^{\odot N-1}
 ] \\ & =  \frac{4}{N^2} \sum_{i=1}^d (u_{i})^{2(N-1)} (u_i^{N} - w_{*,i})^2
\\ &\geq  \frac{4}{N^2} \min_{i\in [d]} (u_{i})^{2(N-1)} \| u^{\odot N} - w_* \|^2_2
 = 4 \min_{i\in [d]} (u_{i})^{2(N-1)}  f(u).
\end{split}
\end{equation}
So it satisfies $\mu$-PL with $\mu = 2 \min_{i\in [d]} (u_{i})^{2(N-1)}$ at $u$.

Now let us consider the average Hessian for the case of $N=2$.
We have
$\nabla^2 f(u) = \text{Diag}( 3 u^{\odot 2} - w_* )$.
The set of global optimal points is
$\mathcal{U}:= \{ \hat{u} \in \reals^d:  \hat{u}_i = \pm \sqrt{w_{*,i}}  \}.$
So the average Hessian between $u$ and $\hat{u}$ is
\begin{equation}
\begin{split}
H_f(u) = \int_0^1 \nabla^2 f( \theta u + (1-\theta) \hat{u} ) d\theta
& =   \int_0^1 
\text{Diag}( 3 ( \theta u + (1-\theta) \hat{u} )^{\odot 2} - w_* ) d\theta
\\ &
=   \text{Diag}( u^{\odot 2} + u \hat{u} + w_* - w_*).
\end{split}
\end{equation}
It is now clear that the smallest eigenvalue of $H_f(u)$ is
 $  \min_{i\in [d]} \left( 
 (u_i)^{2} + u_i \hat{u}_i  \right) $.
 Hence, it satisfies \avg{\lambda_*}{\hat{u}} with the parameter
  $\lambda_* =   \min_{i\in [d]} \left( 
 (u_i)^{2} + u_i \hat{u}_i  \right) $ if $\lambda_*>0$.
 %

We have $\lambda_{\max}( \nabla^2 f(u) ) = \max_{{i \in [d]}} 3 u_i^{2} - w_{*,i}$
so that the gradient is $L=3R^2$-Lipschitz in the ball $\| u \|_{\infty} \leq R$.

Moreover, for $u,z$ in the ball $\| \cdot \|_{\infty} \leq R$, we have
\begin{equation}
\| \nabla^2 f(u) - \nabla^2 f(z) \| 
= \| \text{Diag}(3u^{\odot 2}) -  \text{Diag}(3z^{\odot 2})   \| 
\leq \sqrt{  \sum_{i=1}^d ( 3 u_i^2 - 3 z_i^2 )^2  }
\leq 3 \sqrt{  \sum_{i=1}^d ( | u_i + z_i | | u_i - z_i| )^2  }
\leq 6 R \| u - z \|.
\end{equation}
So the Hessian is $L_H= 6 R$-Lipschitz in the ball $\| \cdot \|_{\infty} \leq R$.
\end{proof}

\subsection{Proof of Lemma~\ref{lem:stay}} \label{app:lem:stay}

\textbf{Lemma~\ref{lem:stay}}
\textit{
Let the initial point $u_0 = \alpha \mathbf{1} \in \reals^{d}_{+}$, where $\alpha$ satisfies $0< \alpha < \min_{i \in [d]} \sqrt{ w_{{*,i}} }$. Denote $L:= ~6 \max_{i \in [d]} w_{*,i}$.
Let $\eta = \frac{c_{\eta}}{L }$, where $c_{{\eta}} \in (0,1]$, and $\beta_{t} \leq \sqrt{ (1-\frac{2c_{\eta}^2 \tilde{c} \alpha^2 }{L} )\left( 1- 4 c_{\mu} \frac{ 2 \alpha^2 }{L} \right) } $ for any $\tilde{c} \in (0,1]$ and $c_{{\mu}} \in (0,\frac{1}{4}]$.
Then, for minimizing \eqref{deep}, the iterate $u_{t}$ generated by HB satisfies $0 < \alpha < u_{{t,i}} <  \sqrt{2 w_{*,i} - \alpha^2 }$ for all $t$ and all $i \in [d]$.
}

\begin{proof}

We observe that the dynamic of HB for minimizing \eqref{deep} is
\begin{equation}
u_{t+1} = u_t \odot \left( 1 + \eta ( w_* - u_t^{\odot 2}  )  \right) + \beta  (u_t - u_{t-1}),
\end{equation}
where $\odot$ represents the element-wise product.
From the dynamic, we see that the dynamic of each dimension $i$ is the same as that of applying HB for minimizing $f_i(u_i):= \frac{1}{4} ( u_{i}^{2} - w_{*,i} )^{2}$.
Therefore, it suffices to fix an $i$ and consider the dynamic of HB for minimizing
\begin{equation}
f_i(u_i):= \frac{1}{4} ( u_{i}^{2} - w_{*,i} )^{2},
\end{equation}
where $f_i(\cdot): \reals \rightarrow \reals$.

Let us recall the Lyapunov function defined on \eqref{lyp}:
$\textstyle
V_{t,i} := f_i(u_{t,i}) - \min_{u_i} f_i(u_i) 
+ \theta ( u_{t,i} - u_{t-1,i} )^2,$
where we have
$f_i(u_i) = \frac{1}{4} ( u_i^{\odot 2} - w_{*,i} )^{2}$ and $\min_{u_i} f_i(u_i) = 0$.
We will let $L:= 6 \max_{i \in [d]} w_{*,i}$ because we will show that $u_{t,i} \leq  \sqrt{2 w_{*,i}}:=R$
and hence by Lemma~\ref{lem:deep}, the function $f(\cdot)$ is $3R^2 = 6 \max_{i \in [d]} w_{*,i}$-smooth for all $\{ u \in \reals^d : \| u \|_{\infty} = R \}$.

Observe that at the beginning, we have
\begin{equation}
V_{0,i} = f_i(u_{0,i} ) = \frac{1}{4} ( \alpha^{2} - w_{*,i} )^2
\end{equation}
since by initialization $u_{0,i} = u_{{-1,i}}$.

The proof is by induction. Initially $t=0$, by Lemma~\ref{lem:deep}, the function value at the iterate $u_{t,i}=\alpha >0$ satisfies the PL condition. 
Therefore, invoking Theorem~\ref{thm:PL}, we have
$V_{1,i} < V_{0,i} = \frac{1}{4} ( \alpha^{2} - w_{*,i} )^2$. 
Since $V_{1,i} < V_{0,i} = \frac{1}{4} ( \alpha^{2} - w_{*,i} )^2$, we know that $0 < \alpha < u_{{1,i}} <  \sqrt{2 w_{*,i} - \alpha^2 } $. In other words, based on Lemma~\ref{lem:deep}, the iterate
$u_{{1,i}}$ stays in a region where the smoothness, PL, and \ao hold.

Now assume that at iteration $t$, $u_{{t,i}}$ satisfies $0 < \alpha < u_{{t,i}} <  \sqrt{2 w_{*,i} - \alpha^2} $. Then, by invoking Theorem~\ref{thm:PL}, we have
$V_{{t+1,i}} \leq V_{{t,i}}$, and hence $V_{{t+1,i}} < V_{{0,i}}= \frac{1}{4} ( \alpha^{2} - w_{*,i} )^2$, which means that 
$0 < \alpha < u_{{t+1,i}} < \sqrt{2 w_{*,i} - \alpha^2} $. This completes the induction and shows that each $u_{t,i}$ of HB stays away from $0$ for all $t$ and all $i \in [d]$.

Finally, 
by Lemma~\ref{lem:deep}, we know the function at any point $u$ such that $u_{i} \geq \alpha, \forall i \in [d]$ satisfies $\mu= 2 \alpha^2$-PL.

\end{proof}
We remark that if initially $u_{0,i} = - \alpha < 0$, where $\alpha < \min_i \sqrt{ w_{{*,i}} }$,
then following the proof of Lemma~\ref{lem:stay}, one can show that $u_{t,i}$ generated by HB with the same condition of $\eta$ and $\beta$ satisfies $0 > \alpha > u_{{t,i}} > - \sqrt{2 w_{*,i} - \alpha^2 } $ for all $t$, which means that HB stays in a region where the smoothness, PL, and \ao hold according to Lemma~\ref{lem:deep}.

\clearpage

\section{\av{\lambda_*} implies $\lambda_*$-PL } \label{app:sinho}

The following is a proof of showing that \av{\lambda_*} implies the PL condition with the same constant $\lambda_*$.
\begin{align}
f(w_*) & = f(w) + \int_0^1 \langle \nabla f ( (1-\tau) w + \tau w_* ) , w_* - w \rangle d \tau \notag
\\ & =  f(w) + \langle \nabla f(w), w_* - w \rangle + 
\int_0^1 
\langle \nabla f ( (1-\tau) w + \tau w_* ) - \nabla f(w)  , w_* - w \rangle  d \tau \notag
\\ & =  f(w) + \langle \nabla f(w), w_* - w \rangle + 
\int_0^1
\left 
\langle
\left(
\int_0^1 \nabla^2 f\left( (1-\alpha) \left( (1-\tau) w + \tau w_* \right) + \alpha w \right)  d\alpha \right) 
\tau (w_* - w) 
 , w_* - w 
\right
 \rangle  d \tau \notag
\\ & \geq  f(w) + \langle \nabla f(w), w_* - w \rangle + 
\int_0^1
\left 
\langle
\lambda_* \tau (w_* - w) 
 , w_* - w 
\right
 \rangle  d \tau \label{sinho}
\\ & = 
f(w) + \langle \nabla f(w), w_* - w \rangle + 
\frac{\lambda_*}{2} \| w_* - w \|^2 \notag 
\\ &
\geq f(w) - \frac{1}{2\lambda_*} \| \nabla f(w) \|^2 \notag,
\end{align}
where \eqref{sinho} uses \av{\lambda_*} and the last ineqaulity is obtained by minimizing
$\langle \nabla f(w), w_* - w \rangle + 
\frac{\lambda_*}{2} \| w_* - w \|^2 $ over~$w$.

\clearpage

\section{Erratum} \label{erratum}

We would like to draw attention of the reader to an error in the proof of Lemma~\ref{vb} in the previous version of this work.

In Appendix~\ref{app:lem2} of the previous version, we wrote:
\begin{mdframed}
\textit{ \small
Proof: By Lemma~1, we have 
\[
\|\Psi_t \|_2 := \|D_tP_t^{-1}P_{t-1}\|_2 \leq \sqrt{\beta_t}\|P_t^{-1}P_{t-1}\|_{2},
\]
$P_{t} = \tilde{U}_t\tilde{P}_t{\rm Diag}(Q_{t,1},\dots,Q_{t,d})$,
where $\tilde{U_t}$ and $\tilde{P_t}$ are unitary matrices.
Since $\tilde{U_t}$ and $\tilde{P_t}$ are unitary matrices,
it does not affect the spectral norm $\| P_t^{-1}P_{t-1}\|_2$. 
So to bound $\| P_t^{-1}P_{t-1}\|_{2}$,
it suffices to analyze the spectrum of the matrix product
${\rm Diag}(Q_{t,1},\dots,Q_{t,d})^{-1} {\rm Diag}(Q_{t-1,1},\dots,Q_{t-1,d})$, which is block diagonal. The eigenvalues of $P_t^{-1}P_{t-1}$ 
are those of ${\rm Diag}(Q_{t,1},\dots,Q_{t,d})^{-1} {\rm Diag}(Q_{t-1,1},\dots,Q_{t-1,d})$.
}
\end{mdframed}

Denote $Q_t:={\rm Diag}(Q_{t,1},\dots,Q_{t,d}) $ and $Q_{t-1}:={\rm Diag}(Q_{t-1,1},\dots,Q_{t-1,d})$.
Then, $\| P_t^{-1}P_{t-1} \|_2 = \| Q_t^{-1} M_t Q_{t-1} \|_2$, where $M_t = (\tilde{P}_t)^{-1} (\tilde{U}_t)^{-1} \tilde{U}_{t-1}\tilde{P}_{t-1}$ is an orthogonal matrix.

However, for the statement that
\textit{...it suffices to analyze the spectrum of the matrix product
${\rm Diag}(Q_{t,1},\dots,Q_{t,d})^{-1} {\rm Diag}(Q_{t-1,1},\dots,Q_{t-1,d})$, ...} to be true, the following co-diagonalization condition needs to hold
\begin{equation} \label{co-diag}
\| Q_t^{-1} M_t Q_{t-1} \|_2  \leq \| Q_t^{-1} Q_{t-1} \|_2, \forall t, 
\end{equation}
which might not be guaranteed in general except for a couple of cases, as the orthogonal matrix $M_t$ is not on the side, it can affect the spectral norm.

\section{Examples when the co-diagonalization $\spadesuit$ hold}
\label{app:when}

An example when the \textit{co-diagonalization} condition $\spadesuit$ \eqref{co-diag} holds is when $M_t$ is the identity matrix, which is the case when the dimension of the optimization variable $w$ of $\min_w f(w)$ is one, i.e., $w \in \reals^d$ and $d=1$.
When the dimension $d=1$,
the matrix $A$ in HB dynamics 
\eqref{dynamic} 
becomes
\begin{equation}
A_t = \begin{bmatrix}
1 - \eta \lambda_t + \beta_t & - \beta_t    \\
1 & 0
\end{bmatrix}
\end{equation}
Observe that $A$ is already in the form of $\Sigma$ in 
\eqref{decompose1},
That is, $\tilde{U}_t\tilde{P}_t$ in the decomposition $A_t = \tilde{U}_t\tilde{P}_t \Sigma \tilde{P}^T_t\tilde{U}^\top_t$ in 
\eqref{decompose1}
is the identity matrix. Hence, Lemma~\ref{vb} holds when the dimension 
of the optimization variable $w$ is one.
More generally, the condition \eqref{co-diag} holds 
when the Hessian is diagonal (e.g., the diagonal network in Example~3).
The is because when the Hessian is diagonal, $\tilde{U}_t$ is the identity matrix, and the permutation matrix $\tilde{P}_t$ is the same, i.e., $\tilde{P}_t=\tilde{P}_{t-1}, \forall t$.

To summarize, Lemma~\ref{vb} requires 
the \textit{co-diagonalization} condition $\spadesuit$, which can be guaranteed in the aforementioned examples but might not hold in general.
Since our acceleration result of HB beyond quadratics in this paper \nocite{wang2022provable} (i.e., Theorem~\ref{thm:meta}) uses Lemma~\ref{vb}, it requires $\spadesuit$. Hence, while this work shows acceleration of HB beyond quadratics when the underlying dimension is $d=1$, identifying when HB has provable acceleration beyond quadratics in a high-dimensional space remains an open question (see e.g., \cite{goujaud2023provable} for some negative (non-convergence) results of HB when $d>1$,
regarding minimizing some particular smooth strongly convex functions).

\end{document}